\documentclass[a4paper,12pt,fleqn]{article}
\usepackage[latin1]{inputenc}
\usepackage[T1]{fontenc}
\usepackage[a4paper,lmargin=4.1cm,textwidth=12.8cm,tmargin=4.9cm,textheight=18.5cm]{geometry}
\usepackage[all,2cell,v2,arrow, curve, frame, matrix, graph, tips]{xy}
\UseAllTwocells
\usepackage[square,numbers,sort]{natbib}
\usepackage{ifthen}
\usepackage{mathbbol}
\usepackage{savesym}
\usepackage{amsmath}
\usepackage{amssymb}
\savesymbol{iint}
\restoresymbol{TXF}{iint}
\usepackage[thmmarks,standard]{ntheorem}
\usepackage{microtype}
\usepackage{stmaryrd}
\usepackage{verbatim}

\newboolean{PourEditeur}
\setboolean{PourEditeur}{false}
\ifthenelse{\boolean{PourEditeur}}{%
}{%
  \usepackage[colorlinks]{hyperref}
}
\usepackage{cleveref}
\usepackage[toc,page]{appendix}
\hypersetup{linkcolor=blue,citecolor=blue,urlcolor=blue}
\theoremstyle{plain}
\crefformat{result}{result~#2#1#3}
\crefformat{proposition}{proposition~#2#1#3}
\crefformat{remark}{remark~#2#1#3}

\makeatletter
\renewcommand*\l@section[2]{%
\ifnum \c@tocdepth >\z@
\addpenalty\@secpenalty
\addvspace{1.0em \@plus\p@}%
\setlength\@tempdima{1.5em}%
\begingroup
\parindent \z@ \rightskip \z@
\parfillskip\z@
\leavevmode \bfseries
\advance\leftskip\@tempdima
\hskip -\leftskip
#1\hfil\hbox{}\par
\endgroup
\fi}
\renewcommand\@dottedtocline[5]{%
\ifnum #1>\c@tocdepth \else
\vskip \z@ \@plus.2\p@
{\leftskip #2\relax \rightskip \z@ \parfillskip -\rightskip
\parindent #2\relax\@afterindenttrue
\interlinepenalty\@M
\leavevmode
\@tempdima #3\relax
\advance\leftskip \@tempdima \null\nobreak\hskip -\leftskip
{#4}\hfill\hbox{}\par
}%
\fi}

\makeatother
\crefformat{result}{result~#2#1#3}
\crefformat{proposition}{proposition~#2#1#3}
\crefformat{remark}{remark~#2#1#3}
\newcommand*{\infG}{\ensuremath{\infty\text{-}\mathbb{G}\text{r}}}
\newcommand*{\infGr}{\ensuremath{\infty\text{-}\mathbb{G}\text{rr}}}
\newcommand*{\ninfG}{\ensuremath{(\infty,n)\text{-}\mathbb{G}\text{r}}}
\newcommand*{\npinfG}{\ensuremath{(\infty,n+1)\text{-}\mathbb{G}\text{r}}}
\newcommand*{\infC}{\ensuremath{\infty\text{-}\mathbb{C}\text{at}}}
\newcommand*{\ninfC}{\ensuremath{(\infty,n)\text{-}\mathbb{C}\text{at}}}
\newcommand*{\minfC}{\ensuremath{(\infty,m)\text{-}\mathbb{C}\text{at}}}
\newcommand*{\zinfC}{\ensuremath{(\infty,0)\text{-}\mathbb{C}\text{at}}}
\newcommand*{\uinfC}{\ensuremath{(\infty,1)\text{-}\mathbb{C}\text{at}}}
\newcommand*{\npinfC}{\ensuremath{(\infty,n+1)\text{-}\mathbb{C}\text{at}}}
\newcommand*{\ninfet}{\ensuremath{(\infty,n)\text{-}\mathbb{E}\text{t}\mathbb{C}\text{at}}}
\newcommand*{\T}{\ensuremath{\mathbb{T}}}

\newcommand*{\G}{\ensuremath{\mathbb{G}\text{r}}}

\newcommand*{\IG}{\ensuremath{\infty\text{-}\G}}

\title{Algebraic Definition of weak $(\infty,n)$-Categories}
\author{Camell Kachour}
\begin{document}
\maketitle
\vspace*{3.5cm}
\begin{abstract}
In this paper we define a sequence of monads $\T^{(\infty,n)} (n\in\mathbb{N})$ on $\IG$, the category of the $\infty$-graphs. 
 We conjecture that algebras for $\T^{(\infty,0)}$, which are defined in a purely algebraic setting, are models of weak $\infty$-groupoids. 
And for all $n\geqslant 1$ we conjecture that algebras for $\T^{(\infty,n)}$, which are defined in a purely algebraic setting, are 
  models of weak $(\infty,n)$-categories.
\end{abstract}

\begin{minipage}{118mm}{\small
}\end{minipage}

\hypersetup{%
     linkcolor=blue}%
\tableofcontents 
\vspace*{1cm}

\vspace*{1cm}

\section*{Introduction}

The notion of weak $(\infty,n)$-category can be made precise in many ways depending on our approach to the higher category. Intuitively this is   
a weak $\infty$-category such that all its cells of dimension greater then $n$ are equivalences. Such an intuition is behind the existing approach to 
weak $(\infty,n)$-categories as developed by Joyal (quasicategories; see \cite{JoyalTierney}), Lurie and Simpson (based on Segal's idea; see 
\cite{LurieTopos,simpson-homotopy-higher}), and Rezk ($\Theta_n$-categories; see \cite{RezkCartesian}). It is already known that all theses definitions 
are equivalent in an appropriate sense. However, all these definitions are not of algebraic nature.

In this paper we propose the first purely algebraic definition of $(\infty,n)$-categories in the globular settings, meaning that we describe this kind of object as algebras of some monad with good categorical properties. We conjecture that the models of the $(\infty,n)$-categories that we propose here, are equivalent to other existing models in a precise sense explained below.

Our main motivation for introducing the algebraic model of $(\infty,n)$-categories came from our wish to build a machinery which would lead to a proof of the  
"Grothendieck conjecture on homotopy types" and, possibly, it generalisation. This conjecture of Grothendieck (see \cite{grothendieck83:_pursuin_stack,bat:monglob}),
 claims that weak $\infty$-groupoids encode all homotopical information about their associated topological spaces. In his seminal article (see \cite{bat:monglob}), 
 Michael Batanin gave the first accurate formulation of this conjecture by building a fundamental weak $\infty$-groupoid functor between the category of 
 topological spaces $\mathbb{T}op$ to the category of the weak $\infty$-groupoids in his sense. This conjecture is not solved yet, and a good direction to solve it
  should be to build first a Quillen Model structure on the category of weak $\omega$-groupoids in the sense of Michael Batanin, and then show that his 
  fundamental weak $\infty$-groupoid functor is a right part of a Quillen equivalence. One obstacle for building such a model structure is that the category of 
  Batanin $\infty$-groupoids is defined in a nonalgebraic way. An important property of the category of weak $\infty$-groupoids 
  $\mathbb{A}lg(\mathbb{T}^{(\infty,0)})$ (see section \ref{Definition}) that we propose here is to be locally presentable 
  (see section \ref{weak-omega-n-categories}). Therefore, we hope that this will allow us in the future to use Smith's theory on combinatorial model 
  categories in our settings (see \cite{Smith-Quillen}). 

 More generally, we expect that it might be possible to build a combinatorial model category structure, for each category $\mathbb{A}lg(\mathbb{T}^{(\infty,n)})$ 
of weak $(\infty,n)$-categories (see section \ref{Definition}) for arbitrary $n\in\mathbb{N}$. As an application we should be able to prove the existence of 
Quillen  equivalences between our $(\infty,n)$-categories and other models of $(\infty,n)$-categories. This should be considered as a generalization of 
the Grothendieck conjecture for higher integers $n>0$.

 The aim of our  present paper is to lay a categorical foundation for this multistage project. The model theoretical aspects of this project will be considered in the future papers (but see the remark \ref{homotopy} about possible approaches).    

Our algebraic description of weak $(\infty,n)$-category is an adaptation of the "philosophy" of categorical stretching as developed by Jacques Penon in \cite{penon1999} 
to describe his weak $\infty$-categories. Here we add the key concept of $(\infty,n)$-reversible structure (see section \ref{reversible-omega-graphs}). 

The plan of this article is as follow :\\
In the first section we introduce \textit{reversors}, which are the operations algebraically describing equivalences. These
operations plus the brilliant idea of categorical stretching developed by Jacques Penon (see \cite{penon1999}) are in the heart of our approach to weak 
$(\infty,n)$-categories.\\ 
The second section introduces the reader to strict $(\infty,n)$-categories,
where we point out the important fact that reversors are "canonical" in the "strict world". Reflexivity for strict $(\infty,n)$-categories is
see as specific structure, using operations that we call \textit{reflexors}, and we study in detail the relationships between
 \textit{reversors} and \textit{reflexors} (see \ref{involutive-reflexivity}). However most material of this section is well known.\\
The third section gives the steps to define our algebraic approach to weak $(\infty,n)$-categories. First 
we define $(\infty,n)$-magmas (see \ref{the-omega-magmas}), which are the "$(\infty,n)$-analogue" of the $\infty$-magmas of Penon. 
Then we define $(\infty,n)$-categorical stretching (see \ref{groupoidal_stretchings}), which is the "$(\infty,n)$-analogue" of the categorical stretching of
Penon. In \cite{penon1999}, Jacques Penon used categorical stretching to weakened strict $\infty$-categories. Roughly speaking, the philosophy
of Penon follows the idea that the "weak" must be controlled by the "strict", and it is exactly what the $(\infty,n)$-categorical stretchings do for
the "$(\infty,n)$-world". Thirdly we give the definition of weak $(\infty,n)$-categories (see \ref{Definition}) as algebras for specifics 
monads $\mathbb{T}^{(\infty,n)}$ on $\infG$. We show in \ref{algebre-magma} that each $\mathbb{T}^{(\infty,n)}$-algebra $(G,v)$ puts on 
$G$ a canonical $(\infty,n)$-magma structure. Then, as we do for the strict case, we study the more subtle 
relationship between \textit{reversors} and \textit{reflexors} for weak $(\infty,n)$-categories (see section \ref{interactions-revers-invo-reflex}). 
Finally in \ref{computations}, we make some computations for weak $\infty$-groupoids. We show that weak $\infty$-groupoids in dimension $1$ are groupoids, and for weak 
$\infty$-groupoids in dimension $2$, their $1$-cells are equivalences.

{\bf Acknowledgement.} 
First of all I am grateful to my supervisors, Michael Batanin and Ross Street, and the important fact that without their support and without their
encouragement, this work could not have been done. I am also grateful to Andr{\'e} Joyal for his invitation to Montr{\'e}al, and for our 
many discussions, which helped me a lot to improve my point of view of higher category theory. I am also grateful to Denis Charles 
Cisinski who shared with me his point of view of many aspects of abstract homotopy theory. I am also grateful to Clemens Berger
for many discussions with me, especially about the non trivial problem of "transfert" for Quillen model categories. I am also grateful to Paul-Andr{\'e} M{\`e}llies who gave
me the chance to talk (in December 2011) about the monad for weak $\infty$-groupoids in his seminar in Paris, and also to Ren{\'e} Guitart and 
Francois M{\'e}tayer who proposed I talk in their seminars. I am also grateful to Christian Lair who shared with me his point 
of view on sketch theory when I was living in Paris. I am also grateful to the categoricians and other mathematicians of our team at Macquarie University
for their kindness, and their efforts for our seminar on Category Theory. Especially I want to mention Dominic Verity, Steve Lack,
Richard Garner, Mark Weber, Tom Booker, Frank Valckenborgh, Rod Yager, Ross Moore, and Rishni Ratman. Also I have a thought for Brian Day who, unfortunately, left us too early. 
Finally I am grateful to Jacques Penon who taught me, many years ago, his point of view on weak $\infty$-categories. 

I dedicate this work to Ross Street.

\section{$(\infty,n)$-Reversible $\infty$-graphs}
 \label{reversible-omega-graphs}
 
 Let $\mathbb{G}$ be the globe category defined as following: For each $m\in\mathbb{N}$, objects of $\mathbb{G}$ are formal objects $\bar{m}$. 
 Morphisms of $\mathbb{G}$ are generated by the formal cosource and cotarget $\xymatrix{\bar{n}\ar[r]<+2pt>^{s^{m+1}_{m}}\ar[r]<-2pt>_{t^{m+1}_{m}}&\overline{m+1}}$ 
 such that we have the following relations $s^{m+1}_{m}s^{m}_{m-1}=s^{m+1}_{m}t^{m}_{m-1}$ and $t^{m+1}_{m}t^{m}_{m-1}=t^{m+1}_{m}s^{m}_{m-1}$. 
 An $\infty$-\textit{graph} $X$ is just a presheaf $\xymatrix{\mathbb{G}^{op}\ar[rr]^{X}&&\mathbb{S}et}$. We denote by $\IG:=[\mathbb{G}^{op},\mathbb{S}et]$ 
 the category of $\infty$-graphs where morphisms are just natural transformations. If $X$ is an $\infty$-graph, sources and targets are still denoted 
 $s^{m+1}_{m}$ and $t^{m+1}_{m}$. If $0\leqslant p<m$ we denote $s^{m}_{p}:=s^{p+1}_{p}\circ ...\circ s^{m}_{m-1}$ and 
 $t^{m}_{p}:=t^{p+1}_{p}\circ ...\circ t^{m}_{m-1}$.    
 
 An $(\infty,n)$-\textit{reversible graph}, or $(\infty,n)$-\textit{graph} for short, is given by a couple $(X, (j^{m}_{p})_{0\leqslant n\leqslant p<m})$ where $X$ is an 
$\infty$-graph (see \cite{penon1999}) or "globular set" (see \cite{bat:monglob}), and $j^{m}_{p}$ are maps ($0\leqslant n\leqslant p<m$), called the \textit{reversors}
     \[\xymatrix{X_{m}\ar[rr]^{j^{m}_{p}}&&X_{m}},\]
such that for all integers $n$, $m$, and $p$ such that $0\leqslant n\leqslant p<m$ we have the following diagrams in $\mathbb{S}et$ which commute serially.

          \[\xymatrix{X_{m}\ar[rr]^{j^{m}_{p}}\ar[d]<+1pt>_{s^{m}_{m-1}}\ar[d]<+5pt>^{t^{m}_{m-1}}&&X_{m}\ar[d]<+3pt>^{s^{m}_{m-1}}\ar[d]<+6pt>_{t^{m}_{m-1}}\\
                                X_{m-1}\ar[rr]^{j^{m-1}_{p}}\ar[d]<+1pt>_{s^{m-1}_{m-2}}\ar[d]<+5pt>^{t^{m}_{m-1}}&&X_{m-1}\ar[d]<+3pt>^{s^{m-1}_{m-2}}\ar[d]<+6pt>_{t^{m}_{m-1}}\\
                                X_{m-2}\ar[rr]^{j^{m-2}_{p}}\ar@{.>}[d]<+1pt>^{}\ar@{.>}[d]<+5pt>^{}&&X_{m-2}\ar@{.>}[d]<+3pt>^{}\ar@{.>}[d]<+6pt>_{}\\
                                X_{p+2}\ar[rr]^{j^{p+2}_{p}}\ar[d]<+1pt>_{s^{p+2}_{p+1}}\ar[d]<+5pt>^{t^{m}_{m-1}}&&X_{p+2}\ar[d]<+3pt>^{s^{p+2}_{p+1}}\ar[d]<+6pt>_{t^{m}_{m-1}}\\
                                X_{p+1}\ar[rr]^{j^{p+1}_{p}}\ar[dr]_{t^{p+1}_{p}}&&X_{p+1}\ar[dl]^{s^{p+1}_{p}}\\
                                &X_{p} } \qquad                                
                                \xymatrix{X_{m}\ar[rr]^{j^{m}_{p}}\ar[d]<+1pt>_{s^{m}_{m-1}}\ar[d]<+5pt>^{t^{m}_{m-1}}&&X_{m}\ar[d]<+3pt>^{s^{m}_{m-1}}\ar[d]<+6pt>_{t^{m}_{m-1}}\\
                                X_{m-1}\ar[rr]^{j^{m-1}_{p}}\ar[d]<+1pt>_{s^{m-1}_{m-2}}\ar[d]<+5pt>^{t^{m}_{m-1}}&&X_{m-1}\ar[d]<+3pt>^{s^{m-1}_{m-2}}\ar[d]<+6pt>_{t^{m}_{m-1}}\\
                                X_{m-2}\ar[rr]^{j^{m-2}_{p}}\ar@{.>}[d]<+1pt>^{}\ar@{.>}[d]<+5pt>^{}&&X_{m-2}\ar@{.>}[d]<+3pt>^{}\ar@{.>}[d]<+6pt>_{}\\
                                X_{p+2}\ar[rr]^{j^{p+2}_{p}}\ar[d]<+1pt>_{s^{p+2}_{p+1}}\ar[d]<+5pt>^{t^{m}_{m-1}}&&X_{p+2}\ar[d]<+3pt>^{s^{p+2}_{p+1}}\ar[d]<+6pt>_{t^{m}_{m-1}}\\
                                X_{p+1}\ar[rr]^{j^{p+1}_{p}}\ar[dr]_{s^{p+1}_{p}}&&X_{p+1}\ar[dl]^{t^{p+1}_{p}}\\
                                &X_{p} } \]

 A morphism of $(\infty,n)$-graphs
 
        \[\xymatrix{(X, (j^{m}_{p})_{0\leqslant n\leqslant p<m})\ar[rr]^{\varphi}&&(X', (j'^{m}_{p})_{0\leqslant n\leqslant p<m}) }\]
is given by a morphism of $\infty$-graphs $\xymatrix{X\ar[r]^{\varphi}&X'}$ which is compatible with the 
  reversors, which means that for integers $0\leqslant n\leqslant p<m$ we have the following commutative squares
  
       \[\xymatrix{X_{m}\ar[d]_{j^{m}_{p}}\ar[rr]^{\varphi_{m}}&&X'_{m}\ar[d]^{j'^{m}_{p}}\\
          X_{m}\ar[rr]_{\varphi_{m}}&&X'_{m} }\]        
The category of $(\infty,n)$-graphs is denoted $(\infty,n)$-$\G$.

\begin{remark}
In \cite{kamelkachour:defalg} we defined the category $(n,\infty)$-$\G$ of "$\infty$-graphes $n$-cellulaires" ($n$-cellular $\infty$-graphs) which is a completely different category from this category $(\infty,n)$-$\G$. The category $(n,\infty)$-$\G$ was used to define an algebraic approach to "weak $n$-higher transformations", still in the same spirit of 
the weak $\infty$-categories of Penon (\cite{penon1999}).
\end{remark}  

\begin{remark}
 Throughout this paper the reversors are denoted by the symbols "$j^{m}_{p}$" except with weak $(\infty,n)$-categories 
 (see comment in section \ref{algebre-magma}) where they are denoted  
 by the symbols "$i^{m}_{p}$". Let us also make a little comment on reflexive $\infty$-graphs (see \cite{penon1999} for their definition). For us a reflexive $\infty$-graph $(X,(1^{p}_{m})_{0\leqslant p<m})$ must be seen as a "structured $\infty$-graph", that is, an $\infty$-graph $X$ equipped with a structure $(1^{p}_{m})_{0\leqslant p<m}$, where the maps 
 $\xymatrix{X(p)\ar[r]^{1^{p}_{m}}&X(m)}$ must be considered as specific operations that we call \textit{reflexors}. 
 Throughout this paper these operations are denoted by the symbols $1^{p}_{m}$ except for 
 the underlying reflexive structure of weak $(\infty,n)$-categories (see section \ref{algebre-magma}) where, instead, they are denoted by the symbols $\iota^{p}_{m}$ (with the Greek letter "iota"). Morphisms between reflexive $\infty$-graphs are morphisms of $\infty$-graphs which respect this structure. In \cite{penon1999} the category of reflexive $\infty$-graphs is denoted $\infGr$. The canonical forgetful functor $\xymatrix{\infGr\ar[r]^{\mathbb{U}}&\infG}$ is a right adjoint, and gives rise to the very important monad $\mathbb{R}$ of reflexive $\infty$-graphs on $\infty$-graphs.
\end{remark}

The reversors are built without using limits, and it is trivial to build the sketch\footnote{see \cite{laircoppey:esquisses,francisborceux:handbook2} for good references on sketch 
theory.} $\mathcal{G}_{n}$ of the $(n,\infty)$-graphs, which has no cones and no cocones, thus $(n,\infty)$-$\G$ is just a category of presheaves, $(n,\infty)$-$\G$ $\simeq$ $[\mathcal{G}_{n};\mathbb{S}et]$. Denote by $\mathcal{G}$ the sketch of $\infty$-graphs. We have the inclusions 
\[\xymatrix{&\mathcal{G}\ar@{^{(}->}[ld]\ar@{^{(}->}[rd]\\
\mathcal{G}_{n+1}\ar@{^{(}->}[rr]&&\mathcal{G}_{n}}\]
which show that the forgetful functor 
\[\xymatrix{\ninfG\ar[r]^(.46){M_{n}}&\npinfG}\] 
which forgets the reversors 
$(j^{m}_{n})_{m\geqslant n+2}$ for each $(\infty,n)$-graphs, has a left and a right adjoint, $L_{n}\dashv M_{n}\dashv R_{n}$. The functor $L_{n}$ is 
 the "free $(\infty,n)$-graphisation functor" on $(\infty,n+1)$-graphs, and the functor $R_{n}$ is the 
"internal $(\infty,n)$-graphisation functor" on $(\infty,n+1)$-graphs. The forgetful functor 
\[\xymatrix{\ninfG\ar[r]^(.54){O_{n}}&\IG}\]
which forgets all the reversors, has a left and a right adjoint, $G_{n}\dashv O_{n}\dashv D_{n}$. The functor $G_{n}$ is the 
"free $(\infty,n)$-graphisation functor" on $\infty$-graphs, and the functor $D_{n}$ is the "internal $(\infty,n)$-graphisation functor" 
on $\infty$-graphs. It is easy to see that $M_{n}$ and $O_{n}$ are monadic because they are conservative and both have 
rights adjoints (so they preserve all coequalizers). 

\section{Strict $(\infty,n)$-categories ($n\in\mathbb{N}$)}
\label{strict-omega-n-categories}

\subsection{Definition}
\label{def-cat}

The definition of the category $\infty$-$\mathbb{C}at$ of $\infty$-categories can be found in \cite{penon1999}.
Let $C$ be a strict $\infty$-category, and for all $0\leqslant p<m$ denote $\circ^{m}_{p}$ its operations.
 These operations are maps 
       \[\xymatrix{\circ^{m}_{p}: C(m)\underset{C(p)}\times C(m)\ar[r]&C(m)}\] 
such that 
 $C(m)\underset{C(p)}\times C(m)=\{(y,x)\in C(m)\times C(m): s^{m}_{p}(y)=t^{m}_{p}(x)\}$.
 
 Let us remind that
 domain and codomain of these operations must satisfy the following conditions: If $(y,x)\in C(m)\underset{C(p)}\times C(m)$, then
 \begin{itemize}
 \item for $0\leqslant p<q<m$ we have
 $s^{m}_{q}(y\circ^{m}_{p}x)=s^{m}_{q}(y)\circ^{q}_{p}s^{m}_{q}(x)$ and $t^{m}_{q}(y\circ^{m}_{p}x)=t^{m}_{q}(y)\circ^{q}_{p}t^{m}_{q}(x)$\\
 \item for $0\leqslant p\leqslant q<m$ we have $s^{m}_{q}(y\circ^{m}_{p}x)=s^{m}_{q}(x)$ and $t^{m}_{q}(y\circ^{m}_{p}x)=t^{m}_{q}(x)$. 
 \end{itemize}
It is the \textit{positional axioms}, following terminology in \cite{penon1999}. 

 If we denote $(C, (1^{p}_{m})_{0\leqslant p<m})$ the underlying reflexivity structure on $C$, then operations $1^{p}_{m}$
  are just an abbreviation for $1^{m-1}_{m}\circ ...\circ 1^{p}_{p+1}$.
 These reflexivity maps $1^{p}_{m}$ are called \textit{reflexors} to point out that we see the reflexivity as a specific structure.

Now consider $\alpha \in C(m)$ be an $m$-cell of $C$. We say that 
$\alpha$ has an $\circ^{m}_{p}$-inverse $(0\leqslant p<m)$ if there is an $m$-cell $\beta \in C(m)$ such that
$\alpha\circ^{m}_{p}\beta=1^{p}_{m}(t^{m}_{p}(\alpha))$ and 
$\beta\circ^{m}_{p}\alpha=1^{p}_{m}(s^{m}_{p}(\alpha))$.

A strict $(\infty,n)$-category $C$ is a strict $\infty$-category such that for all $0\leqslant n\leqslant p<m$, every $m$-cell $\alpha \in C(m)$ has
an $\circ^{m}_{p}$-inverse. If such an inverse exists then it is unique, because it is an inverse for a morphism in a category. 
Thus every strict $(\infty,n)$-category $C$ has an underlying canonical $(\infty,n)$-reversible $\omega$-graph
 $(C, (j^{m}_{p})_{0\leqslant n\leqslant p<m})$ such that the maps $j^{m}_{p}$ give the unique $\circ^{m}_{p}$-inverse for 
 each $m$-cell of $C$. In other words, for each $m$-cell 
 $\alpha$ of $C$ such that $0\leqslant n\leqslant p<m$, we have $\alpha\circ^{m}_{p}j^{m}_{p}(\alpha)=1^{p}_{m}(t^{m}_{p}(\alpha))$ and 
$j^{m}_{p}(\alpha)\circ^{m}_{p}\alpha=1^{p}_{m}(s^{m}_{p}(\alpha))$. Strict $\infty$-functors respect the reversibility. As a matter of fact, consider two strict $(\infty,n)$-categories $C$ and $C'$ and a strict $\infty$-functor $\xymatrix{C\ar[r]^{F}&C'}$. If
 $\alpha$ is an $m$-cell of $C$, then for all $0\leqslant n\leqslant p<m$, we have 
 $F(j^{m}_{p}(\alpha)\circ^{m}_{p}\alpha)=F(j^{m}_{p}(\alpha))\circ^{m}_{p}F(\alpha)= 
  F(1^{p}_{m}(s^{m}_{p}(\alpha)))=1^{p}_{m}(F(s^{m}_{p}(\alpha)))
 =1^{p}_{m}(s^{m}_{p}(F(\alpha)))=j^{m}_{p}(F(\alpha))\circ^{m}_{p}F(\alpha)$ which shows, by the unicity of
 $j^{m}_{p}(F(\alpha))$, that $F(j^{m}_{p}(\alpha))=j^{m}_{p}(F(\alpha))$. Thus morphisms between strict $(\infty,n)$-categories are just 
strict $\infty$-functors. The category of strict $(\infty,n)$-categories, denoted $(\infty,n)$-$\mathbb{C}at$, is a full subcategory 
of $\infty$-$\mathbb{C}at$ because strict $\infty$-functors preserve reversibility.

It is not difficult to see that there is a projective sketch $\mathcal{C}_{n}$ such that there is an equivalence of categories 
$\mathbb{M}od(\mathcal{C}_{n})\simeq \ninfC$. Thus, for all $n\in\mathbb{N}$, 
the category $(\infty,n)$-$\mathbb{C}at$ is locally presentable. 

 Furthermore, for each $n\in\mathbb{N}$, we have the following forgetful functor
 
   \[\xymatrix{\ninfC\ar[r]^(.6){U_{n}}&\infG}.\]
   
  There is an easy inclusion $\mathcal{G}\subset \mathcal{C}_{n}$,
  and this inclusion of sketches produces on passing to models, a functor $C_{n}$ between 
 the category of models $\mathbb{M}od(\mathcal{C}_{n})$ and the category of models $\mathbb{M}od(\mathcal{G})$
 
  \[\xymatrix{\mathbb{M}od(\mathcal{C}_{n})\ar[r]^{C_{n}}&\mathbb{M}od(\mathcal{G})},\] 
and the associated sheaf theorem for sketches of Foltz (see section \cite{foltz}) 
proves that $C_{n}$ has a left adjoint. 
  Thus the following commutative square induced by the previous equivalence of categories
  
  \[\xymatrix{\mathbb{M}od(\mathcal{C}_{n})\ar[r]^{C_{n}}\ar[d]_{\wr}&\mathbb{M}od(\mathcal{G})\ar[d]^{\wr}\\
   \ninfC\ar[r]^(.55){U_{n}}&\infG}\]
produces the required left adjoint
$\xymatrix{F_{n}\dashv U_{n}: \ninfC\ar[r]&
    \infG}$.
      
The unit and the counit of this adjunction are respectively denoted by $\lambda^{(\infty,n)}_{s}$ and $\varepsilon^{(\infty,n)}_{s}$.  
 It is not difficult to show, by using Beck's theorem of monadicity (see for instance \cite{francisborceux:handbook2}) that these functors 
 $U_{n}$ are monadic. This adjunction produces a monad $\mathbb{T}^{(\infty,n)}_{s}=(T^{(\infty,n)}_{s},\mu^{(\infty,n)}_{s},\lambda^{(\infty,n)}_{s})$ on $\infty$-$\G$, which
 is the monad for strict $(\infty,n)$-categories on $\infty$-graphs.  
 \begin{remark}
For each $n\in\mathbb{N}$, when no confusion appears, we will simplify the notation of these monads so that
 $\mathbb{T}^{n}_{s}=(T^{n}_{s},\mu^{n}_{s},\lambda^{n}_{s})$ means $\mathbb{T}^{(\infty,n)}_{s}=(T^{(\infty,n)}_{s},\mu^{(\infty,n)}_{s},\lambda^{(\infty,n)}_{s})$
 on omitting the symbol $\infty$. 
 \end{remark}
\begin{remark}
 We can also define the category $(\infty,\bullet)$-$\mathbb{C}at$ which contains as objects all $(\infty,n)$-categories (where now $n$ can vary)
and as morphisms those of $(\infty,n)$-categories (for all $n\in\mathbb{N}$), and so it is evident that such a category is locally presentable. If $C$ is an object of 
$(\infty,\bullet)$-$\mathbb{C}at$, let us call the \textit{index} of $C$ the integer 
    \[Ind(C)=min\{m\in\mathbb{N}\ \text{ such that } C \text{ is also an object of } \minfC \}.\]
Note that if $\xymatrix{C\ar[r]^{F}&C'}$ is a morphism of $(\infty,n)$-$\mathbb{C}at$ then necessarily: 
\[Ind(C),Ind(C')\leqslant n.\]
 A morphism $\xymatrix{C\ar[r]^{F}&C'}$ in $(\infty,\bullet)$-$\mathbb{C}at$ is given by a strict $\infty$-functor $F$ such that $Ind(C)>Ind(C')$.
This evident characterisation of morphisms in $(\infty,\bullet)$-$\mathbb{C}at$ shows clearly that $(\infty,\bullet)$-$\mathbb{C}at$ is not a full subcategory of 
$\infty$-$\mathbb{C}at$. Let us define the index of $\xymatrix{C\ar[r]^{F}&C'}$ in $(\infty,\bullet)$-$\mathbb{C}at$ as $Ind(F)=Ind(C)$. It is
clear that such a strict $\infty$-functor is a morphism for all categories $(\infty,Ind(F)+r)$-$\mathbb{C}at$ where $r\in\mathbb{N}$. Consider the filtration
\[\xymatrix{\zinfC\ar[r]^{V_{0}}&\uinfC\ar@{.>}[r]^{}&\ar[r]^(.4){V_{n}}\ninfC
 &\npinfC\ar@{.>}[r]&...}\]
such that $(\infty,0)$-$\mathbb{C}at$ is the category of the strict $\infty$-groupoids, $(\infty,1)$-$\mathbb{C}at$ is the category of the strict quasicategories, etc.
  The functors $V_{n}$ ($n\in\mathbb{N}$) involved in this filtration are just inclusions of categories. Thus the filtered colimit of this filtration 
  is just $(\infty,\bullet)$-$\mathbb{C}at$ which
 gives us a more conceptual description.
\end{remark}

 As we did for $(\infty,n)$-graphs (see section \ref{reversible-omega-graphs}) by building functors of "$(\infty,n)$-graphisation",  
  we are going to build some functors of "strict $(\infty,n)$-categorification" by using systematically the Dubuc adjoint triangle theorem 
 (see \cite{dubuc:kanextensionenriched}). 
 
For all $n\in\mathbb{N}$ we have the following triangle in $\mathbb{C}AT$
 
  \[\xymatrix{\ninfC\ar[rrr]^(.45){V_{n}}\ar[rd]_{U_{n}}&&&\ar[dll]^{U_{n+1}}\npinfC\\
 &\infG}\]
where the functor $V_{n}$ forgets the reversors $(j^{m}_{n})_{m\geqslant n+2}$ for each strict $(\infty,n)$-category, and we have the 
 adjunctions $F_{n}\dashv{U_{n}}$ and  $F_{n+1}\dashv{U_{n+1}}$, where in particular
 $U_{n+1}V_{n}=U_{n}$ and $U_{n+1}$ is monadic. So we can apply the Dubuc adjoint triangle theorem 
 which shows that the functor $V_{n}$ has a left adjoint: $L_{n}\dashv{V_{n}}$. 
 For each strict $(\infty,n+1)$-category $C$, the left adjoint $L_{n}$ of $V_{n}$ assigns the free strict $(\infty,n)$-category $L_{n}(C)$ associated to $C$. 
 The functor $L_{n}$ is the "free strict $(\infty,n)$-categorification functor" for strict $(\infty,n+1)$-categories. 
 Notice that the functor $V_{n}$ has an evident right adjoint $R_{n}$. For each strict $(\infty,n+1)$-category $C$, the right adjoint $R_{n}$ of $V_{n}$ 
 assigns the maximal strict $(\infty,n)$-category $R_{n}(C)$ associated to $C$. The proof is trivial because if $D$ is an object of 
 $(\infty,n)$-$\mathbb{C}at$, then the unit map $\xymatrix{D\ar[r]^(.3){\eta_n}&R_{n}(V_{n}(D))}$ is just the identity $1_{D}$, and its 
 universality becomes straightforward.
 
We can apply the same argument to following triangle in $\mathbb{C}AT$ (where here the functor $V$  forgets all the reversors or can be seen 
 as an inclusion) 
 
  \[\xymatrix{\ninfC\ar[rrr]^{V}\ar[rd]_{U_{n}}&&&\ar[dll]^{U}\infC\\
 &\infty-\G}\] 
to prove that the functor $V$ has a left adjoint: $L\dashv{V}$. For each strict $\infty$-category $C$, the left adjoint $L$ of $V$ assigns the free strict 
 $(\infty,n)$-category $L_{n}(C)$ associated to $C$. The functor $L$ is the "free strict $(\infty,n)$-categorification functor" for strict $\infty$-categories.
 Notice also that the functor $V$ has an evident right adjoint $R$, with a trivial argument as before, for the adjunction $V_{n}\dashv R_{n}$.

%
%
%
%
%
 
  \begin{remark}     
  The previous functors $V_{n}$, $V$ are from our point 
  of view not only inclusions but are also "trivial forgetful functors". 
Indeed for instance, they occur in the paper \cite{metayerara:browngolasinski} where they don't see
 strict $\infty$-groupoids (which are in our terminology $(\infty,0)$-categories) as strict $\infty$-categories equipped with 
 canonical reversible structures. So from their point of view $V$ is just an inclusion. We don't claim that their point of view is incorrect
 but we believe that our point of view, which is more algebraic (the reversors $j^{m}_{p}$ must be
 seen as unary operations), gives more clarity by showing that this inclusion is also 
 a forgetful functor which forgets the canonical and unique reversible structures of some specific strict $\infty$-categories. 
 Basically in our point of view, a strict $(\infty,n)$-category ($n\in\mathbb{N}$) is a strict
 $\infty$-category equipped with some canonical specific structures.
\end{remark}

 \subsection{$(\infty,n)$-Involutive structures and $(\infty,n)$-reflexivity structures.}
 \label{involutive-reflexivity}
  
 The involutive properties and the reflexive structures (see below) are important properties or structures that are part of
 each strict $(\infty,n)$-categories ($n\in\mathbb{N}$). We could have spoken about strict $(\infty,n)$-categories without referring to these
 two specific structures, but we believe that it is informative to especially point out that these two structures are canonical in the world of  
  strict $(\infty,n)$-categories but are not canonical in the world of weak $(\infty,n)$-categories (see section \ref{interactions-revers-invo-reflex}). 
   In particular we will show that they cannot be weakened for
    weak $(\infty,n)$-categories, but only for some specific equalities which are part of these two kinds of structures 
    (see section \ref{weak-omega-n-categories}). This part can also be seen as the observation that some other properties or structures which are true  in the world of
     strict $(\infty,n)$-categories might or not be weakened in the world of weak $(\infty,n)$-categories. 
   
 An involutive $(\infty,n)$-graph is given by an $(\infty,n)$-graph $(X, (j^{m}_{p})_{0\leqslant n\leqslant p<m})$ such that $j^{m}_{p}\circ  
 j^{m}_{p}=1_{X_{m}}$. The involutive $(\infty,n)$-graphs form a full reflexive subcategory $i(\infty,n)$-$\G$ of $(\infty,n)$-$\G$. For each $n\in\mathbb{N}$ and for each
 strict $(\infty,n)$-category $C$, its underlying $(\infty,n)$-graph $(C, (j^{m}_{p})_{0\leqslant n\leqslant p<m})$ is an  
 involutive $(\infty,n)$-graph. As a matter of fact, for each $0\leqslant n\leqslant p<m$ and for each $m$-cell $\alpha\in C(m)$, we have
 $j^{m}_{p}(j^{m}_{p}(\alpha))\circ^{m}_{p}j^{m}_{p}(\alpha)=1^{p}_{m}(s^{m}_{p}(j^{m}_{p}(\alpha)))=1^{p}_{m}(t^{m}_{p}(\alpha))$, 
 thus $j^{m}_{p}(j^{m}_{p}(\alpha))$ is the unique $\circ^{m}_{p}$-inverse of $j^{m}_{p}(\alpha)$ which is $\alpha$. Thus 
 $j^{m}_{p}(j^{m}_{p}(\alpha))=\alpha$. 
     
 A reflexive $(\infty,n)$-graph is given by a triple $(X, (1^{p}_{m})_{0\leqslant p<m}, (j^{m}_{p})_{0\leqslant n\leqslant p<m})$ 
 where $(X, (1^{p}_{m})_{0\leqslant p<m})$ is an $\infty$-graph equipped with a reflexivity structure $(1^{p}_{m})_{0\leqslant p<m}$, 
 $(X, (j^{m}_{p})_{0\leqslant n\leqslant p<m})$ is an
 $(\infty,n)$-graph, and such that we have the following commutative diagram in $\mathbb{S}et$, which expresses
 the relations between the truncation at level $n$ of the reflexors $1^{p}_{m}$ and the reversors
 $j^{m}_{p}\hspace{.1cm}(0\leqslant n\leqslant p<m)$.  
 
 \[\xymatrix{X_{n}\ar[rr]^{j^{n}_{p}}&&X_{n}\\
     X_{n-1}\ar[u]^{1^{n-1}_{n}}\ar[rr]^{j^{n-1}_{p}}&&X_{n-1}\ar[u]_{1^{n-1}_{n}}\\
     X_{n-2}\ar[u]^{1^{n-2}_{n-1}}\ar[rr]^{j^{n-2}_{p}}&&X_{n-2}\ar[u]_{1^{n-2}_{n-1}}\\
     X_{p+2}\ar@{.>}[u]^{}\ar[rr]^{j^{p+2}_{p}}&&X_{p+2}\ar@{.>}[u]^{}\\
                                X_{p+1}\ar[u]^{1^{p+1}_{p+2}}\ar[rr]^{j^{p+1}_{p}}&&X_{p+1}\ar[u]_{1^{p+1}_{p+2}}\\
                                &X_{p}\ar[lu]^{1^{p}_{p+1}}\ar[ru]_{1^{p}_{p+1}} } \]
Thus for all integers $p, q, m$ such that $0\leqslant n\leqslant q < m$ and $q\geqslant p\geqslant 0$ we have
  $j^{m}_{q}(1^{p}_{m}(\alpha))=1^{p}_{m}(\alpha)$ and for all integers $p, q, m$ such that $0\leqslant n\leqslant q < p < m$ we
  have $j^{m}_{q}(1^{p}_{m}(\alpha))=1^{p}_{m}(j^{p}_{q}(\alpha))$. Morphisms between
  reflexive $(\infty,n)$-graphs are those which are morphisms of reflexive $\infty$-graphs and
 morphisms of $(\infty,n)$-graphs. The category of reflexive $(\infty,n)$-graphs is denoted by  
  $(\infty,n)$-$\mathbb{G}\text{rr}$. 
  
  For each $n\in\mathbb{N}$ and for each strict $(\infty,n)$-category $C$, its underlying $(\infty,n)$-graph $(C$ $, (j^{m}_{p})_{0\leqslant n\leqslant p<m})$ has an underlying   
 reflexive $(\infty,n)$-graph $(C, (1^{p}_{m})_{0\leqslant p<m}, (j^{m}_{p})_{0\leqslant n\leqslant p<m})$ where $(1^{p}_{m})_{0\leqslant p<m}$
 is the reflexive structure of the underlying strict $\infty$-category of $C$. As a matter of fact for all integers $p, q, m$ such that 
 $0\leqslant n\leqslant q < p < m$ we have  
 $j^{m}_{q}(1^{p}_{m}(\alpha))\circ^{m}_{q}1^{p}_{m}(\alpha)=1^{q}_{m}(s^{m}_{q}(1^{p}_{m}(\alpha)))=1^{q}_{m}(s^{p}_{q}(\alpha))$.
 Also we have the following axiom for strict $\infty$-categories: If $q<p<m$ and $s^{p}_{q}(y)=t^{p}_{q}(x)$ then
 $1^{p}_{m}(y\circ^{p}_{q}x)=1^{p}_{m}(y)\circ^{m}_{q}1^{p}_{m}(x)$ (see \cite{penon1999}). But here we have  
 $s^{p}_{q}(j^{p}_{q}(\alpha))=t^{q+1}_{q}s^{q+2}_{q+1}...s^{p}_{p-1}(\alpha)=t^{p}_{q}(\alpha)$, thus we can apply this axiom:
  $1^{p}_{m}(j^{p}_{q}(\alpha))\circ^{m}_{q}1^{p}_{m}(\alpha)=1^{p}_{m}(j^{p}_{q}(\alpha)\circ^{p}_{q}\alpha)
  =1^{p}_{m}(1^{q}_{p}(s^{p}_{q}(\alpha)))=1^{q}_{m}(s^{p}_{q}(\alpha))$ which 
  shows that $1^{p}_{m}(j^{p}_{q}(\alpha))$ is the unique $\circ^{m}_{q}$-inverse of $1^{p}_{m}(\alpha)$ and thus
  $1^{p}_{m}(j^{p}_{q}(\alpha))=j^{m}_{q}(1^{p}_{m}(\alpha))$. Also for all integers $p, q, m$ such that $0\leqslant n\leqslant q < m$ and $q\geqslant p\geqslant 0$ we have
  $j^{m}_{q}(1^{p}_{m}(\alpha))\circ^{m}_{q}1^{p}_{m}(\alpha)=1^{q}_{m}(s^{m}_{q}(1^{p}_{m}(\alpha)))=1^{q}_{m}(1^{p}_{q}(\alpha))=1^{p}_{m}(\alpha)$
  and $1^{p}_{m}(\alpha)\circ^{m}_{q}1^{p}_{m}(\alpha)=1^{p}_{m}(\alpha)$ because $q\geqslant p$, thus $1^{p}_{m}(\alpha)$ is
  the unique $\circ^{m}_{q}$-inverse of $1^{p}_{m}(\alpha)$ and thus $1^{p}_{m}(\alpha)=j^{m}_{q}(1^{p}_{m}(\alpha))$.

  As in section \ref{strict-omega-n-categories}, it is not difficult to show some similar results for the category $i(\infty,n)$-$\G$
   and the category $(\infty,n)$-$\mathbb{G}\text{rr}$ ($n\in\mathbb{N}$):
    
\begin{itemize}
  \item For each $n\in\mathbb{N}$, the categories $i(\infty,n)$-$\G$ and $(\infty,n)$-$\mathbb{G}\text{rr}$ are both locally presentable.
  \item For each $n\in\mathbb{N}$, there is a monad $\mathbb{I}^{(\infty,n)}_{i}=(I^{(\infty,n)}_{i},\mu^{(\infty,n)}_{i},\lambda^{(\infty,n)}_{i})$ 
  on $\infty$-$\G$ ($i$ here means "involutive") such that $\mathbb{A}lg(\mathbb{I}^{(\infty,n)}_{i})\simeq i(\infty,n)$-$\G$, and a
  monad $\mathbb{R}^{(\infty,n)}_{r}=(R^{(\infty,n)}_{r},\mu^{(\infty,n)}_{r},\lambda^{(\infty,n)}_{r})$ 
  on $\infty$-$\G$ ($r$ here means "reflexive") such that $\mathbb{A}lg(\mathbb{R}^{(\infty,n)}_{r})\simeq (\infty,n)$-$\mathbb{G}\text{rr}$.  
 \item We can also consider the category $i(\infty,n)$-$\mathbb{G}\text{rr}$ of involutive $(\infty,n)$-graphs equipped with a specific 
 reflexivity structure, whose morphisms are those of $(\infty,n)$-$\G$ which respect the reflexivity structures. This category $i(\infty,n)$-$\mathbb{G}\text{rr}$
  is also locally presentable. For each $n\in\mathbb{N}$, there is a monad $\mathbb{K}^{(\infty,n)}_{ir}=(K^{(\infty,n)}_{ir},\mu^{(\infty,n)}_{ir},\lambda^{(\infty,n)}_{ir})$ 
  on $\infty$-$\G$ ($ir$ means here "involutive-reflexive") such that $\mathbb{A}lg(\mathbb{K}^{(\infty,n)}_{ir})\simeq i(\infty,n)$-$\mathbb{G}\text{rr}$. 
  Also there is a forgetful functor from the category $i(\infty,n)$-$\mathbb{G}\text{rr}$ to the category $(\infty,n)$-$\mathbb{G}\text{rr}$ which has a left adjoint, 
  the functor "$(\infty,n)$-involution" of any reflexive $(\infty,n)$-graph, and there is a forgetful functor from the category $i(\infty,n)$-$\mathbb{G}\text{rr}$ 
  to the category $i(\infty,n)$-$\G$ which has a left adjoint, the functor "$(\infty,n)$-reflexivisation" of any
  involutive $(\infty,n)$-graph. These left adjoints are built by using the Dubuc adjoint triangle theorem. 
\end{itemize}

 \section{Weak $(\infty,n)$-categories ($n\in\mathbb{N}$)}
 \label{weak-omega-n-categories}
 
 In this chapter we are going to define our algebraic point of view of weak $(\infty,n)$-categories for all $n\in\mathbb{N}$.
 As the reader will see, many kind of filtrations as in section \ref{strict-omega-n-categories} could be studied here, because
 their filtered colimits do exist. But we have avoided that,
 because all the filtrations involved here are not built with "inclusion functors" but are all right adjunctions, and the author
 has not found a good description of their corresponding filtered colimits. We do hope to afford it in a future work because
 we believe that these filtered colimits have their own interest in abstract homotopy theory, and also in higher category theory. 
  
\subsection{$(\infty,n)$-Magmas}
\label{the-omega-magmas}  

The definition of $\infty$-magmas and morphisms between $\infty$-magmas can be found in 
\cite{kamelkachour:defalg, penon1999}. Roughly speaking an $\infty$-magma in the sense 
of Penon is a reflexive $\infty$-graph equipped with composition $\circ^{m}_{p}$ which satisfy
only \textit{positional axioms} as in \ref{def-cat}. Let us call
 $\infty$-$\mathbb{M}ag$ the category of $\infty$-magmas.
  An $(\infty,n)$-magma is an $\infty$-magma such that its underlying reflexive $\infty$-graph is equipped 
 with a specific $(\infty,n)$-reversible structure in the sense of \ref{reversible-omega-graphs}. However
 an $(\infty,n)$-magma might have several $(\infty,n)$-reversible structures.
  
\begin{remark}
The reversibility part of an $(\infty,n)$-magma has no relation with its reflexivity structure, neither with the involutive properties, contrary to the strict $(\infty,n)$-categories where
their reversible structures, their involutive structures and their reflexivity structures are all related together (see \ref{involutive-reflexivity}). Instead we are going to see
in this section \ref{weak-omega-n-categories}, that each underlying $(\infty,n)$-categorical stretching of any weak $(\infty,n)$-category ($n\in\mathbb{N}$) 
is especially going to be weakened, for the specific relation between the reversibility structure and the involutive structure, inside its underlying reflexive 
$(\infty,n)$-reversible $\infty$-magma the equalities $j^{n+1}_{n}\circ j^{n+1}_{n}=1_{M_{n+1}}$. Also we are going to see
in \ref{groupoidal_stretchings}, that each $(\infty,n)$-categorical stretching is especially going to be weakened, for the specific relation between the 
reversibility structure and the reflexibility structure, inside its underlying reflexive $(\infty,n)$-reversible $\infty$-magma
the equalities $j^{m}_{q}\circ1^{m-1}_{m}=1^{m-1}_{m}\circ j^{m-1}_{q}$ and the equalities $j^{m}_{m-1}\circ1^{m-1}_{m}=1^{m-1}_{m}$.

We believe that for the other equalities involving the reversible structures, the involutive structures and the reflexivity structures (see \ref{involutive-reflexivity}) for the 
strict $(\infty,n)$-categories ($n\in\mathbb{N}$), instead, a cylinder object (as in \cite{metayer:folk}) should appear between $j^{m}_{p}\circ  
 j^{m}_{p}$ and $1_{X_{m}}$, between $j^{m}_{q}\circ 1^{p}_{m}$ and $1^{p}_{m}\circ j^{p}_{q}$, and between $j^{m}_{p}\circ j^{m}_{p}$ and $1_{X_{m}}$, for those 
 underlying $(\infty,n)$-magmas of the $(\infty,n)$-categorical stretchings.
  \end{remark}

The basic examples of $(\infty,n)$-magmas are the strict $(\infty,n)$-categories.  
 Let us denote by $\mathbb{M}=(M,(j^{m}_{p})_{0\leqslant n\leqslant p<m})$ 
 and $\mathbb{M}'=(M',(j'^{m'}_{p'})_{0\leqslant n\leqslant p'<m'})$ two 
 $(\infty,n)$-magmas where $M$ and $M'$ are respectively their underlying $\infty$-magmas, and 
 $(j^{m}_{p})_{0\leqslant n\leqslant p<m}$ and  $(j'^{m'}_{p'})_{0\leqslant n\leqslant p'<m'}$ are respectively their
 underlying reversors. A morphism between these $(\infty,n)$-magmas 
 
                        \[\xymatrix{\mathbb{M}\ar[rr]^{\varphi}&&\mathbb{M}'}\]
is given by their underlying morphism of $\infty$-magmas
  
          \[\xymatrix{M\ar[rr]^{\varphi}&&M'}\] 
such that $\varphi$ preserves the $(\infty,n)$-reversible structure, which means that for integers $0\leqslant n\leqslant p<m$ we have the following commutative squares
  
       \[\xymatrix{M_{m}\ar[d]_{j^{m}_{p}}\ar[rr]^{\varphi_{m}}&&M'_{m}\ar[d]^{j'^{m}_{p}}\\
          M_{m}\ar[rr]_{\varphi_{m}}&&M'_{m} }\]   
The category of the $(\infty,n)$-magmas is denoted by $(\infty,n)$-$\mathbb{M}ag$ and it is evident to see that it is not a
 full subcategory of $\infty$-$\mathbb{M}ag$.
 
  As in section \ref{strict-omega-n-categories}, it is not difficult to show the following similar results for the $(\infty,n)$-magmas ($n\in\mathbb{N}$):
    
\begin{itemize}
  \item For each $n\in\mathbb{N}$, the category $(\infty,n)$-$\mathbb{M}ag$ is locally presentable.
  \item For each $n\in\mathbb{N}$, there is a monad $\mathbb{T}^{(\infty,n)}_{m}=(T^{(\infty,n)}_{m},\mu^{(\infty,n)}_{m},\lambda^{(\infty,n)}_{m})$ 
  on $\infty$-$\G$ ($m$ means here "magmatic") such that $\mathbb{A}lg(\mathbb{T}^{(\infty,n)}_{m})\simeq (\infty,n)-\mathbb{M}ag$.  
 \item By using Dubuc's adjoint triangle theorem we can build functors of "$(\infty,n)$-magmatifications" similar to those in section \ref{strict-omega-n-categories}.
\end{itemize}

\begin{remark}
\label{homotopy}
Let us explain some informal intuitions related to homotopy. The reader can notice that we can imagine many variations of "$\infty$-magmas" similar to those of \cite{penon1999}, or those that we propose in this paper (see above), but which still need to keep the presence of "higher symmetries", encoded by the reversors (see the section \ref{reversible-omega-graphs}), or in a bit more hidden way, by the reflexors plus some compositions $\circ^{m}_{p}$ (see section \ref{involutive-reflexivity}). For instance we can build kinds of "$\infty$-magma", their adapted "stretchings" (similar to those of section \ref{groupoidal_stretchings}), and their corresponding 
"weak $\infty$-structures" (similar to those of section \ref{Definition}). All that just by using reversors, reflexors plus compositions. Such variations of "higher structures" must be all the time projectively sketchables (see section \ref{strict-omega-n-categories}). If we restrict to take the models of such sketches in $\mathbb{S}et$, then these categories must be locally presentables and equipped with an interesting Quillen model structure. The Smith theorem could bring simplification to proving these intuitions. For instance the authors in \cite{metayer:folk} have built a folk Quillen model structure on $\omega$-$\mathbb{C}at$, by using the Smith theorem, and $\omega$-$\mathbb{C}at$ is such a "higher structure" weak equivalences were build only with reflexors and compositions. So, even though the goal of this paper is to give an algebraic approach of the weak $(\infty,n)$-categories, we believe that such structures and variations may provide us many categories with interesting Quillen model structure. Our slogan is : "enough reversors, and (or) reflexors plus some higher compositions" capture enough symmetries for doing abstract homotopy theory, based on higher category theory.
\end{remark}

 \subsection{$(\infty,n)$-Categorical stretchings}
 \label{groupoidal_stretchings}
 
 Now we are going to define $(\infty,n)$-categorical stretchings $(n\in\mathbb{N})$, which are for the weak
 $(\infty,n)$-categories what categorical stretchings 
 (see \cite{kamelkachour:defalg, penon1999}) are for weak $\infty$-categories, and we are going to use these important 
 tools to weaken the axioms of strict $(\infty,n)$-categories.
 In this paragraph the category of the categorical stretching of \cite{penon1999}
   is denoted $\infty$-$\mathbb{E}t\mathbb{C}at$.

 An $(\infty,n)$-categorical stretching is given by a categorical stretching $\mathbb{E}^{n}=(M^{n},C^{n},\pi^{n},([-,-]_{m})_{m\in\mathbb{N}})$ such that
 $M^{n}$ is an $(\infty,n)$-magma, $C^{n}$ is a strict $(\infty,n)$-category, $\pi^{n}$ is a morphism of $(\infty,n)$-$\mathbb{M}ag$, and 
 $([-,-]_{m})_{m\in\mathbb{N}})$ is an extra structure. If $m\geqslant 1$, two $m$-cells $c_1,c_0$ of $M^{n}$ are parallels if 
 $t^{m}_{m-1}(c_{1})=t^{m}_{m-1}(c_{0})$ and if $s^{m}_{m-1}(c_{1})=s^{m}_{m-1}(c_{0})$. In that case we denote it $c_1\|c_0$. 
 For the convenience of the reader we are going to recall the "bracketing structure" 
 $([-,-]_{m})_{m\in\mathbb{N}}$, which is the key structure of the Penon approach for weakened the axioms of strict 
 $\infty$-categories. Also we voluntary use the same notations as in \cite{penon1999}:
              \[([-,-]_{m}:\xymatrix{\widetilde{M}_{m}\ar[r]^{}&M_{m+1}})_{m\in\mathbb{N}}\]
 is a sequence of maps, where
 \[\widetilde{M}_{m}=\{(c_{1},c_{0})\in M_{m}\times M_{m}: c_1\|c_0\hspace{.1cm}\text{and}\hspace{.1cm}\pi_{n}(c_{1})=\pi_{n}(c_{0})\}\] 
and such that 
  \begin{itemize} 
  \item $\forall (c_{1},c_{0})\in\widetilde{M}_{m}$, $t^{m}_{m-1}([c_{1},c_{0}]_{m})=c_{1}$, $s^{m}_{m-1}([c_{1},c_{0}]_{m})=c_{0}$,
   \item $\pi_{n+1}([c_{1},c_{0}]_{m})=1^{m-1}_{m}(\pi_{m}(c_{1}))=1^{m-1}_{m}(\pi_{m}(c_{0}))$,
   \item $\forall n\in M_{m}, [c,c]_{m}=1^{m}_{m+1}(c)$.
 \end{itemize}
A morphism of $(\infty,n)$-categorical stretchings,

            \[\xymatrix{\mathbb{E}\ar[rr]^{(m,c)}&&\mathbb{E}'}\]
is given by the following commutative square in $(\infty,n)$-$\mathbb{M}ag$,
 
  \[\xymatrix{M\ar[d]_{\pi}\ar[rr]^{m}&&M'\ar[d]^{\pi'}\\
          C\ar[rr]_{c}&&C' }\]
such that $\forall m\in \mathbb{N}$, $\forall (c_{1},c_{0})\in\widetilde{M}_{m}$

     \[m_{m+1}([c_{1},c_{0}]_{m})=[m_{m}(c_{1}),m_{m}(c_{0})]_{m}\]
The category of the $(\infty,n)$-categorical stretchings is denoted $(\infty,n)$-$\mathbb{E}t\mathbb{C}at$.

As in section \ref{strict-omega-n-categories}, it is not difficult to show the following similar results for $(\infty,n)$-categorical stretchings ($n\in\mathbb{N}$):
    
\begin{itemize}
  \item For each $n\in\mathbb{N}$, the category $(\infty,n)$-$\mathbb{E}t\mathbb{C}at$ is locally presentable (see also \ref{Definition}).  
 \item By using Dubuc's adjoint triangle theorem we can build functors of "$(\infty,n)$-categorisation stretching" for any $(\infty,n+1)$-categorical stretching, and
 for any categorical stretching.
\end{itemize}
 
 \subsection{Definition}
 \label{Definition}
 
 Let us write $\mathbb{T}^{P}=(T^{P},\mu^{P},\lambda^{P})$ for Penon's monad on the $\infty$-graphs for weak $\infty$-categories.
  
 For each  $n\in\mathbb{N}$ consider the forgetful functors
 \[\xymatrix{\ninfet\ar[r]^(.6){U_{n}}&\infG}\]
 given by $\xymatrix{(M,C,\pi,([,])_{m\in\mathbb{N}})\ar@{|->}[r]^{}&M }$
 
 Also, for each $n\in\mathbb{N}$ the categories $\ninfet$ and $\infty$-$\G$ are sketchable (in section \ref{strict-omega-n-categories} we call $\mathcal{G}$
 the sketch of $\infty$-graphs). Let us call $\mathcal{E}_{n}$ the sketch of $(\infty,n)$-categorical stretchings. These sketches are both projective and 
 there is an easy inclusion $\mathcal{G}\subset \mathcal{E}_{n}$. This inclusion of sketches produces, in passing to models, a functor $W_{n}$ between 
 the category of models $\mathbb{M}od(\mathcal{E}_{n})$ and the category of models $\mathbb{M}od(\mathcal{G})$:
 
  \[\xymatrix{\mathbb{M}od(\mathcal{E}_{n})\ar[r]^{W_{n}}&\mathbb{M}od(\mathcal{G})}\] 
and the associated sheaf theorem for sketches of Foltz (\cite{foltz}) proves that $W_{n}$ has a left adjoint. 
 Furthermore we show that $\mathbb{M}od(\mathcal{E}_{n})\simeq (\infty,n)\text{-}\mathbb{E}t\mathbb{C}at$  
  is an equivalence of categories. Thus the following commutative square induced by these equivalences
  
  \[\xymatrix{\mathbb{M}od(\mathcal{E}_{n})\ar[r]^{W_{n}}\ar[d]_{\wr}&\mathbb{M}od(\mathcal{G})\ar[d]^{\wr}\\
   (\infty,n)\text{-}\mathbb{E}t\mathbb{C}at\ar[r]^(.55){U_{n}}&\IG}\]
produces the required left adjoint $F_{n}$ of $U_{n}$
  
  \[\xymatrix{\ninfet\ar[r]<+5pt>^(.6){U_{n}}&
    \infG\ar[l]<+4pt>^(.4){F_{n}}_(.4){\top}}\] 
The unit and the counit of this adjunction are respectively denoted $\lambda^{(\infty,n)}$ and $\varepsilon^{(\infty,n)}$.

 This adjunction generates a monad $\mathbb{T}^{(\infty,n)}=(T^{(\infty,n)},\mu^{(\infty,n)},\lambda^{(\infty,n)})$ on $\IG$.
 
 \begin{definition}
For each $n\in\mathbb{N}$, a weak $(\infty,n)$-category is an algebra for the monad $\mathbb{T}^{(\infty,n)}=(T^{(\infty,n)},\mu^{(\infty,n)},\lambda^{(\infty,n)})$ on $\IG$.
 \end{definition}

 \begin{remark}
\label{simplification_des_notations}
For each $n\in\mathbb{N}$, when no confusion occurs, we will simplify the notation of these monads:
 $\mathbb{T}^{n}=(T^{n},\mu^{n},\lambda^{n})=\mathbb{T}^{(\infty,n)}=(T^{(\infty,n)},\mu^{(\infty,n)},\lambda^{(\infty,n)})$,
 by omitting the symbol $\infty$. 
\end{remark} 
 For each
$n\in\mathbb{N}$, the category $\mathbb{A}lg(\mathbb{T}^{(\infty,n)})$ is locally presentable. As a matter of fact, the adjunction
$\xymatrix{\ninfet\ar[r]<+5pt>^(.6){U_{n}}&\infG\ar[l]<+4pt>^(.4){F_{n}}_(.4){\top}}$ involves
the categories $\ninfet$ and $\infty$-$\G$ which are both accessible (because they are both projectively sketchable thus locally presentable). But the
forgetful functor $U_{n}$ has a left adjoint, thus thanks to the proposition 5.5.6 of \cite{francisborceux:handbook2}, it preserves filtered colimits. 
Thus the monad $\mathbb{T}^{n}$ preserves filtered colimits in the locally presentable category $\infty$-$\G$, and the theorem 5.5.9 
of \cite{francisborceux:handbook2} implies that the category $\mathbb{A}lg(\mathbb{T}^{(\infty,n)})$ is locally presentable as well.

Now we are going to build some functors of "weak $(\infty,n)$-categorification" by using systematically Dubuc's adjoint triangle theorem 
 (see \cite{dubuc:kanextensionenriched}). 
 For all $n\in\mathbb{N}$ we have the following triangle in $\mathbb{C}AT$
\[\xymatrix{\mathbb{A}lg(\mathbb{T}^{n})\ar[rrr]^(.45){V_{n}}\ar[rd]_{U_{n}}&&&\ar[dll]^{U_{n+1}}\mathbb{A}lg(\mathbb{T}^{n+1})\\
 &\infG}\] 
The functors $V_{n}$ can be thought of as forgetful functors which 
 forgets the reversors $(i^{m}_{n})_{m\geqslant n+2}$ for each weak $(\infty,n)$-category
 (see \ref{algebre-magma} for the definition of the reversors produced by each weak $(\infty,n)$-category).
 We have the 
 adjunctions $F_{n}\dashv{U_{n}}$ and  $F_{n+1}\dashv{U_{n+1}}$, where in particular
 $U_{n+1}V_{n}=U_{n}$ and $U_{n+1}$ is monadic. So we can apply Dubuc's adjoint triangle theorem 
 (see \cite{dubuc:kanextensionenriched}) to show that the functor $V_{n}$ has a left adjoint: $L_{n}\dashv{V_{n}}$. 
 For each weak $(\infty,n+1)$-category $C$, the left adjoint $L_{n}$ of $V_{n}$ yields the free weak $(\infty,n)$-category $L_{n}(C)$ associated to $C$. 
 $L_{n}$ can be seen as the "free weak $(\infty,n)$-categorification functor" for weak $(\infty,n+1)$-categories.  
 
 We can apply the same argument to the following triangles in $\mathbb{C}AT$ (where here the functor $V$  forgets all the reversors) 
 
  \[\xymatrix{\mathbb{A}lg(\mathbb{T}^{n})\ar[rrr]^{V}\ar[rd]_{U_{n}}&&&\ar[dll]^{U}\mathbb{A}lg(\mathbb{T}^{P})\\
 &\infG}\] 
to prove that the functor $V$ has a left adjoint: $L\dashv{V}$. For each weak $\infty$-category $C$, the left adjoint $L$ of $V$ builds the free weak 
 $(\infty,n)$-category $L_{n}(C)$ associated to $C$. $L$ is the "free weak $(\infty,n)$-categorification functor" for weak $\infty$-categories.

 \subsection{Magmatic properties of weak $(\infty,n)$-categories ($n\in\mathbb{N}$)}
 \label{algebre-magma}
  
If $(G,v)$ is a $\mathbb{T}^{n}$-algebra then $\xymatrix{G\ar[r]^(.45){\lambda^{n}_{G}}&\T^{n}(G)}$
 is the associated universal map and $\xymatrix{M^{n}(G)\ar[r]^{\pi^{n}_{G}}&C^{n}(G)}$ is the free $(\infty,n)$-categorical 
 stretching associated to $G$, and we write $(\star^{m}_{p})_{0\leqslant p<m}$ for the composition laws of $M^{n}(G)$. 
Also let us define the following composition laws on $G$: If $a,b\in G(m)$ are such that $s^{m}_{p}(a)=t^{m}_{p}(b)$ then we put
$a\circ^{m}_{p}b=v_{m}(\lambda^{n}_{G}(a)\star^{m}_{p}\lambda^{n}_{G}(b))$, if $a\in G(p)$ then we put $\iota^{p}_{m}(a):=v_{m}(1^{p}_{m}(\lambda^{n}_{G}(a)))$,
and if $a\in G(m)$ and $0\leqslant n\leqslant p<m$ then we put $i^{m}_{p}(a):=v_{m}(j^{m}_{p}(\lambda^{n}_{G}(a)))$. It is easy to show that with these
definitions, the $\mathbb{T}^{n}$-algebra $(G,v)$ puts on $G$ an $(\infty,n)$-magma structure.

In \cite{penon1999} the author showed that if $a,b$ are $m$-cells of $\T^{n}(G)$ such that $s^{m}_{p}(a)=t^{m}_{p}(b)$ then 
$v_{m}(a\star^{m}_{p}b)=v_{m}(a)\circ^{m}_{p}v_{m}(b)$. We are going to show that if $a$ is a $p$-cell of $\T^{n}(G)$ such that
$0\leqslant n\leqslant p<m$ then $v_{m}(1^{p}_{m}(a))=\iota^{p}_{m}(v_{p}(a))$ and if $a$ is an $m$-cell of $T^{n}(G)$ such that $0\leqslant n\leqslant p<m$
then $v_{m}(j^{m}_{p}(a))=i^{m}_{p}(v_{m}(a))$. In other words, each underlying morphism 
$\xymatrix{\T^{n}(G)\ar[r]^(.6){v}&G}$ in $\infty$-$\G$ of a weak $(\infty,n)$-category $(G,v)$ is also a morphism of 
$(\infty,n)$-$\mathbb{M}ag$ when we consider as equipped with the $(\infty,n)$-magmatic structures that we have defined above.
Proofs of these magmatic properties became standard after the work in \cite{kamelkachour:defalg,penon1999}, but for the
comfort of the reader we are going to give complete proof.

All reversors for algebras are denoted "$i^{m}_{p}$" because there is no risk of confusion. The $\mathbb{T}^{n}$-algebra $(\T^{n}(G),\mu^{n}_{G})$ on
$\T^{n}(G)$ is an $(\infty,n)$-reversible structure $(i^{m}_{p})_{0\leqslant n\leqslant p<m}$ such that for all $t$ in $\T^{n}(G)(m)$ we have 
$j^{m}_{p}(t)=i^{m}_{p}(t)$. As a matter of fact
          $i^{m}_{p}(t):=\mu^{n}_{G}(j^{m}_{p}(\lambda^{n}_{T^{n}(G)}(t)))=j^{m}_{p}(\mu^{n}_{G}(\lambda^{n}_{T^{n}(G)}(t)))$
because $\mu^{n}_{G}$ forgets that a morphism preserves the involutions, so 
$i^{m}_{p}(t)=j^{m}_{p}(t)$. Furthermore a morphism of $\T^{n}$-algebras
$\xymatrix{(G,v)\ar[r]^{f}&(G',v')}$ is such that for all $t\in G(m)$ with $0\leqslant n\leqslant p<m$ we have 
$f(i^{m}_{p}(t))=i^{m}_{p}(f(t))$. Indeed 
$f(i^{m}_{p}(t))=f(v_{m}(j^{m}_{p}(\lambda^{n}_{G}(t))))=v'_{m}(T^{n}(f)(j^{m}_{p}(\lambda^{n}_{G}(t))))
=v'_{m}(j^{m}_{p}(T^{n}(f)(\lambda^{n}_{G}(t))))$
because $\T^{n}(f)$ forgets that a morphism preserves the reversors thus:
$f(i^{m}_{p}(t))=v'_{m}(j^{m}_{p}(T^{n}(f)(\lambda^{n}_{G}(t))))=v'(j^{m}_{p}(\lambda^{n}_{G'}(f(t))$. Thus, 
because a $\T^{n}$-algebra $(G,v)$ is determines a morphism of $\T^{n}$-algebras:
     \[\xymatrix{(\mu^{n}_G,\T^{n}(G))\ar[r]^{v}&(G,v)}\]
we have the useful formula $v_{m}(j^{m}_{p}(t))=i^{m}_{p}(v_{m}(t))$.

All reflexors for algebras are denoted "$\iota^{p}_{m}$" because there is no risk of confusion, and we use the symbols "$1^{p}_{m}$" for reflexors coming from the free
$(\infty,n)$-categorical stretchings. The $\mathbb{T}^{n}$-algebra $(\T^{n}(G),\mu^{n}_{G})$ put on $\T^{n}(G)$ a reflexive structure $(\iota^{p}_{m})_{0\leqslant p<m}$
such that for all $t$ in $\T^{n}(G)(p)$ we have $1^{p}_{m}(t)=\iota^{p}_{m}(t)$. As a matter of fact
  $\iota^{p}_{m}(t):=\mu^{n}_{G}(1^{p}_{m}(\lambda^{n}_{\T^{n}(G)}(t)))=1^{p}_{m}(\mu^{n}_{G}(\lambda^{n}_{\T^{n}(G)}(t)))$ because
$\mu^{n}_{G}$ forgets that a morphism preserves the reflexivities, so $\iota^{p}_{m}(t)=1^{p}_{m}(t)$.           
 Furthermore a morphism of $\T^{n}$-algebras
$\xymatrix{(G,v)\ar[r]^{f}&(G',v')}$ is such that for all $t\in G(p)$ with $0\leqslant p<m$, we have 
$f(\iota^{p}_{m}(t))=\iota^{p}_{m}(f(t))$. Indeed
$f(\iota^{p}_{m}(t))=f(v_{m}(1^{p}_{m}(\lambda^{n}_{G}(t))))=v'_{m}(\T^{n}(f)(1^{p}_{m}(\lambda^{n}_{G}(t))))
=v'_{m}(1^{p}_{m}(T^{n}(f)(\lambda^{n}_{G}(t))))$ because $\T^{n}(f)$ forgets that a morphism preserves the reflexors,
so $f(\iota^{p}_{m}(t))=v'(1^{p}_{m}(\T^{n}(f)(\lambda^{n}_{G}(t))))=v'(1^{p}_{m}(\lambda^{n}_{G'}(f(t))))=\iota^{p}_{m}(f(t))$. 
Thus, because a $\T^{n}$-algebra $(G,v)$ is also a morphism of $\T^{n}$-algebras:
     \[\xymatrix{(\mu^{n}_G,\T^{n}(G))\ar[r]^(.6){v}&(G,v)}\]
thus we have the useful formula $v_{m}(1^{p}_{m}(t))=\iota^{p}_{m}(v_{m}(t))$. 
 
\subsection{Interactions between reversibility structures, involutive structures, and reflexivity structures}
\label{interactions-revers-invo-reflex}
 
 The reversors for strict $(\infty,n)$-categories and for $(\infty,n)$-magmas are denoted by "j", whereas the reversors
 for weak $(\infty,n)$-categories are denoted by "i". Let us fix an $n\in\mathbb{N}$ and a strict $(\infty,n)$-category $C$. 
 We know that in $C$ (see section \ref{strict-omega-n-categories}) we have for each
 $0\leqslant n \leqslant p<m$ the involutive properties $j^{m}_{p}\circ j^{m}_{p}=1_{X_{m}}$. However for 
 weak $(\infty,n)$-categories this property does not even up to coherence cell; yet reversors of type $i^{m+1}_{m}$ do permit to weakened version of the  
  involutive property. As a matter of fact consider now a weak $(\infty,n)$-category $(G,v)$,
  the free $(\infty,n)$-categorical stretching $\xymatrix{M^{n}(G)\ar[r]^{\pi^{n}_G}&C^{n}(G)}$ associated to $G$  
  and the universal map
   $\xymatrix{G\ar[r]^(.45){\lambda^{n}_G}&\T^{n}(G)}$.
  For each $\alpha\in G(m+1)$ 
 we have $i^{m+1}_{m}(i^{m+1}_{m}(\alpha))=i^{m+1}_{m}(v(j^{m+1}_{m}(\lambda^{m}_G(\alpha))))=v(j^{m+1}_{m}(j^{m+1}_{m}(\lambda^{m}_G(\alpha))))$
 because the algebra $(G,v)$ preserves the reversibility structure (see \ref{computations}). It shows that $i^{m+1}_{m}(i^{m+1}_{m}(\alpha))\|\alpha$ because
 \begin{align*}
 s^{m+1}_{m}(i^{m+1}_{n}(i^{m+1}_{m}(\alpha)))&=s^{m+1}_{m}(v_{m+1}(j^{m+1}_{m}(j^{m+1}_{m}(\lambda^{m}_G(\alpha)))))\\
 &=v_{m}(s^{m+1}_{m}(j^{m+1}_{m}(j^{m+1}_{m}(\lambda^{m}_G(\alpha)))))\\
 &=v_{m}(t^{m+1}_{m}(j^{m+1}_{m}(\lambda^{m}_G(\alpha))))\\
 &=v_{m}(s^{m+1}_{m}(\lambda^{m}_G(\alpha)))\\
 &=v_{m}(\lambda^{m}_G(s^{m+1}_{m}(\alpha)))\\
 &=s^{m+1}_{m}(\alpha), 
 \end{align*}
 and similarly we show
 that $t^{m+1}_{m}(i^{m+1}_{m}(i^{m+1}_{m}(\alpha)))=t^{m+1}_{m}(\alpha)$. 
  
 But also in the free $(\infty,n)$-categorical
  stretching associated to $G$, which control the algebricity of $(G,v)$, creates between the $m+1$-cells $j^{m+1}_{m}(j^{m+1}_{m}(\lambda^{m}_G(\alpha)))$
  and $\lambda^{m}_G(\alpha)$ an $(m+2)$-cell of coherence: 
  \[[j^{m+1}_{m}(j^{m+1}_{m}(\lambda^{m}_G(\alpha)));\lambda^{m}_G(\alpha)]_{m+1}.\] 
Thus at the level of algebra it shows that
  there is a coherence cell between $i^{m+1}_{m}(i^{m+1}_{m}(\alpha))$ and $\alpha$.  
  
 The other equalities
 $j^{m}_{p}\circ j^{m}_{p}=1_{X_{m}}$ $(0\leqslant n<p<m)$ which are valid in any strict $(\infty,n)$-category have no reasons
 to be weakened in any weak $(\infty,n)$-category for the simple reasons that the axioms of $(\infty,n)$-reversibility structure
 do not imply parallelism between the $m$-cells $i^{m}_{p}(i^{m}_{p}(\alpha))$ and $\alpha$ when $p>n$. However
 we believe that for any $m$-cell $\alpha$ of any weak $(\infty,n)$-category $(G,v)$, if $0\leqslant n<p<m$, then there exists between
 the $m$-cells $i^{m}_{p}(i^{m}_{p}(\alpha))$ and $\alpha$, a cylinder object in the sense of \cite{metayer:folk}. 
 
Now let us fix an $n\in\mathbb{N}$ and a strict $(\infty,n)$-category $C$. We know that for each $p$-cell $\alpha$ 
of $C$ and for each $0\leqslant n\leqslant q<p<m$ we have the equalities
 $j^{m}_{q}(1^{p}_{m}(\alpha))=1^{p}_{m}(j^{p}_{q}(\alpha))$ but also for each $0\leqslant n\leqslant q<m$ and $0\leqslant p\leqslant q$ we have the equalities
 $j^{m}_{q}(1^{p}_{m}(\alpha))=1^{p}_{m}(\alpha)$ (see section \ref{strict-omega-n-categories}). However for a 
 weak $(\infty,n)$-category $(G,v)$ and for any $p$-cell $\alpha$ in it, the $(\infty,n)$-reversibility structure, for each $0\leqslant n\leqslant q<p<m$,
  doesn't ensure the parallelism between the $m$-cells $i^{m}_{q}(\iota^{p}_{m}(\alpha))$ and $\iota^{p}_{m}(i^{p}_{q}(\alpha))$, and 
  for each $0\leqslant n\leqslant q<m$ and $0\leqslant p\leqslant q$, doesn't ensure the parallelism between the $m$-cells $i^{m}_{q}(\iota^{p}_{m}(\alpha))$ 
  and $\iota^{p}_{m}(\alpha)$.
  
 Thus these equalities which are true in the strict case are not necessarily weakened in the weak case. However there are 
  certain situations where in the weak case these equalities are replaced by some coherence cells. As a matter of fact if now $(G,v)$ is
  a weak $(\infty,n)$-category then it is easy to check, thanks to the axioms for the $(\infty,n)$-reversibility structure (see section \ref{reversible-omega-graphs}) that
  if $p=m-1$ and $0\leqslant n\leqslant q<m-1$, then for any $(m-1)$-cell $\alpha$ of $(G,v)$ we have 
  $i^{m}_{q}(\iota^{m-1}_{m}(\alpha))\|\iota^{m-1}_{m}(i^{m-1}_{q}(\alpha))$. Indeed, we have 
 \begin{align*}
  i^{m}_{q}(\iota^{m-1}_{m}(\alpha))&=i^{m}_{q}(v_{m}(1^{m-1}_{m}(\lambda^{n}_{G}(\alpha))))\\
  &=v_{m}(j^{m}_{q}(1^{m-1}_{m}(\lambda^{n}_G(\alpha)))).
 \end{align*}
Thus
  \begin{align*}
  s^{m}_{m-1}(i^{m}_{q}(\iota^{m-1}_{m}(\alpha)))&=s^{m}_{m-1}(v_{m}(j^{m}_{q}(1^{m-1}_{m}(\lambda^{n}_{G}(\alpha)))))\\
  &=v_{m-1}(s^{m}_{m-1}(j^{m}_{q}(1^{m-1}_{m}(\lambda^{n}_{G}(\alpha)))))\\  
  &=v_{m-1}(j^{m-1}_{q}(s^{m}_{m-1}(1^{m-1}_{m}(\lambda^{n}_{G}(\alpha)))))\\ 
  &=v_{m-1}(j^{m-1}_{q}(\lambda^{n}_{G}(\alpha)))\\
  &=i^{m-1}_{q}(\alpha)\\
  &=v_{m-1}(\lambda^{n}_{G}(i^{m-1}_{q}(\alpha)))\\
  &=v_{m-1}(s^{m}_{m-1}(1^{m-1}_{m}(\lambda^{n}_{G}(i^{m-1}_{q}(\alpha)))))\\
  &=s^{m}_{m-1}(v_{m}(1^{m-1}_{m}(\lambda^{n}_{G}(i^{m-1}_{q}(\alpha)))))\\
  &=s^{m}_{m-1}(\iota^{m-1}_{m}(i^{m-1}_{q}(\alpha))),
 \end{align*}
  and similarly we see that $t^{m}_{m-1}(i^{m}_{q}(\iota^{m-1}_{m}(\alpha)))=t^{m}_{m-1}(\iota^{m-1}_{m}(i^{m-1}_{q}(\alpha)))$. 
But also in the free $(\infty,n)$-categorical stretching associated to $G$, which control the algebricity of $(G,v)$, 
between the $m$-cells $j^{m}_{q}(1^{m-1}_{m}(\lambda^{n}_G(\alpha)))$ and
$1^{m-1}_{m}(j^{m-1}_{q}(\lambda^{n}_{G}(\alpha)))$, an $(m+1)$-cell of coherence is created:
 \[[j^{m}_{q}(1^{m-1}_{m}(\lambda^{n}_G(\alpha)));1^{m-1}_{m}(j^{m-1}_{q}(\lambda^{n}_{G}(\alpha)))]_{m}.\]
  Thus at the level of algebras it shows
 that there is a coherence cell between $i^{m}_{q}(\iota^{m-1}_{m}(\alpha))$ and $\iota^{m-1}_{m}(i^{m-1}_{q}(\alpha))$.
 
Furthermore, thanks to the $(\infty,n)$-reversibility structure (see section \ref{reversible-omega-graphs}) we easily prove that for $p=q=m-1$ we have for any
  $(m-1)$-cells $\alpha$ of any weak $(\infty,n)$-category $(G,v)$ the parallelism 
  $i^{m}_{m-1}(\iota^{m-1}_{m}(\alpha))\|\iota^{m-1}_{m}(\alpha)$. As a matter of fact
  \begin{align*}
  i^{m}_{m-1}(\iota^{m-1}_{m}(\alpha))&=i^{m}_{m-1}(v_{m}(1^{m-1}_{m}(\lambda^{n}_{G}(\alpha))))\\
  &=v_{m}(j^{m}_{m-1}(1^{m-1}_{m}(\lambda^{n}_{G}(\alpha)))), 
 \end{align*}
Thus
 \begin{align*}
  s^{m}_{m-1}(i^{m}_{m-1}(\iota^{m-1}_{m}(\alpha)))&=s^{m}_{m-1}(v_{m}(j^{m}_{m-1}(1^{m-1}_{m}(\lambda^{n}_{G}(\alpha)))))\\
  &=v_{m-1}(s^{m}_{m-1}(j^{m}_{m-1}(1^{m-1}_{m}(\lambda^{n}_{G}(\alpha)))))\\
  &=v_{m-1}(t^{m}_{m-1}(1^{m-1}_{m}(\lambda^{n}_{G}(\alpha))))\\
  &=v_{m-1}(\lambda^{n}_{G}(\alpha))\\
  &=\alpha\\
  &=s^{m}_{m-1}(\iota^{m-1}_{m}(\alpha)), 
 \end{align*}
and similarly we show that 
  $t^{m}_{m-1}(i^{m}_{m-1}(1^{m}_{m-1}(\alpha)))=t^{m}_{m-1}(1^{m}_{m-1}(\alpha))$.
But also in the free $(\infty,n)$-categorical
  stretching associated to $G$, which controls the algebricity of $(G,v)$, between the $m$-cells $j^{m}_{m-1}(1^{m-1}_{m}(\lambda^{n}_{G}(\alpha)))$
  and $1^{m-1}_{m}(\lambda^{n}_{G}(\alpha))$ creates an $(m+1)$-cell of coherence 
  \[[j^{m}_{m-1}(1^{m-1}_{m}(\lambda^{n}_{G}(\alpha)));1^{m-1}_{m}(\lambda^{n}_{G}(\alpha))]_{m}.\] 
Thus  at the level of algebras it shows that there is a coherence cell between $i^{m}_{m-1}(\iota^{m-1}_{m}(\alpha))$ and $\iota^{m-1}_{m}(\alpha)$.  
  
  For the other equalities $j^{m}_{q}(1^{p}_{m}(\alpha))=1^{p}_{m}(j^{p}_{q}(\alpha))$ (for $0\leqslant n\leqslant q<p<m$ and $p\neq m-1$) and 
  $j^{m}_{q}(1^{p}_{m}(\alpha))=1^{p}_{m}(\alpha)$ (for $0\leqslant n\leqslant q<m$, $0\leqslant p\leqslant q$ and $p,q\neq m-1$), 
  which are valid in any strict $(\infty,n)$-category, they have no reasons to be weakened in any weak $(\infty,n)$-category for the simple reason that 
  the axioms of the $(\infty,n)$-reversibility structure doesn't imply the parallelism between these $m$-cells $i^{m}_{q}(\iota^{p}_{m}(\alpha))$ and 
  $\iota^{p}_{m}(i^{p}_{q}(\alpha))$, or between the $m$-cells $i^{m}_{q}(\iota^{p}_{m}(\alpha))$ and $\iota^{p}_{m}(\alpha)$. However
 we believe that for any weak $(\infty,n)$-category $(G,v)$, between such $m$-cells $i^{m}_{q}(\iota^{p}_{m}(\alpha))$ and 
  $\iota^{p}_{m}(i^{p}_{q}(\alpha))$ and between such $m$-cells $i^{m}_{q}(\iota^{p}_{m}(\alpha))$ and $\iota^{p}_{m}(\alpha)$,
 they are cylinder objects in the sense of \cite{metayer:folk}.   
  
\subsection{Computations in dimension 1 and in dimension 2}
\label{computations}

We are going to see that in dimension $1$, algebras for the monad $\mathbb{T}^{0}$ (see reference \ref{simplification_des_notations} for this notation) of the weak 
$(\infty,0)$-categories (see section \ref{weak-omega-n-categories}), commonly called in the literature weak $\infty$-groupoids, are  
usual groupoids, and in dimension $2$, algebras for this monad give rise to bigroupoids where in particular we are
are going to show that each $1$-cell are equivalences. First let us recall some basic definitions
that we can find in \cite{kamelkachour:defalg}. A reflexive $\infty$-graph has dimension $p\in \mathbb{N}$ if all its $q$-cells for which
$q>p$ are identity cells and if there is at least one $p$-cell which is not an identity cell. Thus
 reflexive $\infty$-graphs of dimension $0$ are just sets.
 An $(\infty,0)$-categorical stretching $\mathbb{E}^{0}=(M^{0},C^{0},\pi^{0},([,])_{m\in\mathbb{N}})$ (that we can also call "groupoidal stretching" by analogy with the
 "categorical stretchings" of Penon) 
has dimension $p\in\mathbb{N}$ if the underlying reflexive $\infty$-graph of $M$ has dimension $p$. 
A $\mathbb{T}^{0}$-algebra $(G;v)$ has dimension $p\in\mathbb{N}$ if $G$ has dimension 
$p$ when $G$ is considered with its canonical reflexivity structure (see \ref{algebre-magma}). 

\subsection{Dimension 1}

 \begin{proposition}
 Let $(G;v)$ be a $\mathbb{T}^{0}$-algebra of dimension $1$ and let 
 $\xymatrix{a\ar[r]^{f}&b}$ be a $1$-cell of $G$. Then $f$ is an $\circ^{1}_{0}$-isomorphism.
\end{proposition}

\begin{proof}

 Let us denote by $\xymatrix{G\ar[r]^(.4){\lambda^{0}_G}&\mathbb{T}^{0}(G)}$ the universal map associated to $G$ and by 
  $\xymatrix{M^{0}(G)\ar[r]^{\pi^{0}_G}&C^{0}(G)}$ the free groupoidal stretching associated to $G$. First
  we are going to show that in $M^{0}(G)$ lives a diagram of the type
 \[
      \xymatrix{\lambda^{0}_G(f)\star^{1}_{0}j^{1}_{0}\bigl(\lambda^{0}_G(f)\bigr)
      \ar@{=>}[r]^(.6){\beta}&1^{0}_{1}\bigl(\lambda^{0}_G(b)\bigl)}
 \] 
Since $\pi_G$ is a morphism of $\infty$-magmas, we have  that
\begin{align*}
 \pi^{0}_G \bigl(\lambda^{0}_G(f)\star^{1}_{0}j^{1}_{0}\bigl(\lambda^{0}_G(f)\bigr)\bigr)
        &=\pi^{0}_G\bigl(\lambda^{0}_G(f)\bigr)\circ^{1}_{0}\pi_G\bigl(j^{1}_{0}\bigl(\lambda^{0}_G(f)\bigr)\bigr)\\
        &=\pi^{0}_G\bigl(\lambda^{0}_G(f)\bigr)\circ^{1}_{0}j^{1}_{0}\bigl(\pi^{0}_G\bigl(\lambda^{0}_G(f)\bigr)\bigr)\\
          &=1^{0}_{1}\bigl(t^{1}_{0}\bigl(\pi_G\bigl(\lambda^{0}_G(f)\bigl)\bigl)\bigl)\\
        &=1^{0}_{1}\bigl(\pi_G\bigl(\lambda^{0}_G(b)\bigr)\bigr)
        =\pi^{0}_G\bigl(1^{0}_{1}\bigl(\lambda^{0}_G(b)\bigr)\bigr)\;,      
 \end{align*}
where the second equality holds because $\pi^{0}_G$ respects the $(\infty,0)$-reversible structures,
and the third equality holds because $\pi^{0}_G \bigl(\lambda^{0}_G(f) \bigr)$ 
and $j^{1}_{0}\pi^{0}_G\bigl(\lambda^{0}_G(f) \bigr)$ are $\circ^{1}_{0}$-inverse
within the strict $\infty$-groupoid $C^{0}(G)$. Thus by the contractibility structure in $M^{0}(G)$, we get the following coherence $2$-cell in $M^{0}(G)$ 
 \[
 \xymatrix{\lambda^{0}_G(f)\star^{1}_{0}j^{1}_{0}(\lambda^{0}_G(f))
      \ar@{=>}[rrrrr]^(.55){\beta=[\lambda^{0}_G(f)\star^{1}_{0}j^{1}_{0}(\lambda^{0}_G(f));1^{0}_{1}(\lambda^{0}_G(b))]_{1}}
      &&&&&1^{0}_{1}\bigl(\lambda^{0}_G(b)\bigr)}
 \]      
  By applying to it the $\mathbb{T}^{0}$-algebra $(G;v)$ we obtain the following $2$-cell in $G$
 \[
     \xymatrix{
     v\bigl(\lambda^{0}_G(f)\star^{1}_{0}j^{1}_{0}\bigl(\lambda^{0}_G(f)\bigr)\bigr)
      \ar@{=>}[rr]^(.55){v(\beta)}
      &&v\bigl(1^{0}_{1}\bigl(\lambda^{0}_G(b)\bigr)\bigr)}
 \]  
 with
 \[
        v\bigl(\lambda^{0}_G(f)\star^{1}_{0}j^{1}_{0}\bigl(\lambda^{0}_G(f)\bigr)\bigr)
        =v\bigl(\lambda^{0}_G(f)\bigr)\circ^{1}_{0}v\bigl(j^{1}_{0}\bigl(\lambda^{0}_G(f)\bigr)\bigr)\\
        = f\circ^{1}_{0}i^{1}_{0}(f)
 \]
because $v$ is a morphism of $\infty$-magmas.  

 Recall that we have $\iota^{0}_{1}(b):=v\bigl(1^{0}_{1}\bigl(\lambda^{0}_G(b)\bigr)\bigr)$. 
 Thus we obtain the following $2$-cell in $G$:
 \[
 \xymatrix{f\circ^{1}_{0}i^{1}_{0}(f)
       \ar@{=>}[rr]^(.55){v(\beta)}
      &&\iota^{0}_{1}(b)\;.}
\]        
   But the $\mathbb{T}^{0}$-algebra $(G;v)$ has dimension $1$, thus
 $v(\beta)$ is an identity, which shows that  $f\circ^{1}_{0}i^{1}_{0}(f)=\iota^{0}_{1}(b)$. 
 
 By the same method we can prove that   $i^{1}_{0}(f)\circ^{1}_{0}f=\iota^{0}_{1}(a)$. 
 Thus $f$ is an $\circ^{1}_{0}$-isomorphism 
\end{proof}

\subsection{Dimension 2}
\begin{proposition}
 Let $(G;v)$ be a $\mathbb{T}^{0}$-algebra of dimension $2$ and let
 $\xymatrix{a\ar[r]^{f}&b}$ be a $1$-cell of $G$. Then $f$ is an equivalence.
\end{proposition}
\begin{proof}
 Actually we are going to exhibit a diagram in $G$ of the following form 
\[
 \xygraph{
{a}="p1" [r(3)] {b}="p2"
"p1":@`{"p1"+(-1,1), "p1"+(-1,-1)}"p1"|-{}="pcod1"^(.5){\iota^{0}_{1}(a)}
"p1":@`{"p1"+(-2,2), "p1"+(-2,-2)}"p1"|-{}="pdom1"_-{i^{1}_{0}(f) \circ^{0}_{1} f}
"pdom1":@{}"pcod1"|(.35){}="d2cell1"|(.65){}="c2cell1" "d2cell1":@{=>}"c2cell1"^{v(\alpha)}
"p2":@`{"p2"+(1,1), "p2"+(1,-1)}"p2"|-{}="pcod2"_(.5){\iota^{0}_{1}(b)}
"p2":@`{"p2"+(2,2), "p2"+(2,-2)}"p2"|-{}="pdom2"^-{f \circ^{0}_{1} i^{1}_{0}(f)}
"pdom2":@{}"pcod2"|(.35){}="d2cell2"|(.65){}="c2cell2" "d2cell2":@{=>}"c2cell2"_{v(\beta)}
"p1":@/^{1pc}/"p2"^-{f} "p2":@/^{1pc}/"p1"^-{i^{1}_{0}(f)} 
} 
\]
and show that the $2$-cells $v(\alpha)$ and $v(\beta)$ are $\circ^{2}_{1}$-isomorphism.
 Let us denote by $\xymatrix{G\ar[r]^(.4){\lambda^{0}_G}&\mathbb{T}^{0}(G)}$ the universal map associated to $G$, and by 
  $\xymatrix{M^{0}(G)\ar[r]^{\pi^{0}_G}&C^{0}(G)}$ the free groupoidal stretching associated to $G$. 

Consider the following $2$-cell in $M^{0}(G)$
\[\beta= \bigl[\lambda^{0}_G(f)\star^{1}_{0}j^{1}_{0}\bigl(\lambda^{0}_G(f) \bigr)
                       \,;\, 1^{0}_{1} \bigl(\lambda^{0}_G(b) \bigr)\bigr]_{1}^{}
\]    
 We are going to show that in $M^{0}(G)$ lives a diagram of the type
{\small
\[  
 \xygraph{
    [] {\bigl[ \lambda^{0}_G(f)\star^{1}_{0}j^{1}_{0}\bigl(\lambda^{0}_G(f)\bigr)\,;\, 
     1^{0}_{1}\bigl(\lambda^{0}_G(b)\bigr)\bigr]_{1}
 \star^{2}_{1}
 j^{2}_{1}\bigl(\bigl[
   \lambda^{0}_G (f)\star^{1}_{0} j^{1}_{0}\bigl(\lambda^{0}_G(f)\bigr)
    \,;\, 1^{0}_{1}\bigl(\lambda^{0}_G(b)\bigr)\bigr]_{1}^{} \bigr)}
 :@3 [d]
   {1^{0}_{2}\bigl(\lambda^{0}_G(b)\bigr)} ^{\;\lambda^{0}_{f}}
 }
 \]}%
\noindent
 Because $\pi^{0}_G$ is a morphism of $\infty$-magmas, we have 
 \begin{align*}
        \pi^{0}_G\bigl(&\bigl[
         \lambda^{0}_G (f)\star^{1}_{0}j^{1}_{0}\bigl(\lambda^{0}_G (f)\bigr)\,;\, 
          1^{0}_{1}\bigl(\lambda^{0}_G(b)\bigr)\bigr)]_{1}
     \\ &\qquad\qquad 
     \star^{2}_{1} j^{2}_{1}\bigl(\bigl[
    \lambda^{0}_G(f)\star^{1}_{0}j^{1}_{0}\bigl(\lambda^{0}_G(f)\bigr)\,;\, 
     1^{0}_{1}\bigl(\lambda^{0}_G(b)\bigr)\bigr]_{1}^{}\bigr)\bigr)
\\
     &=
     \pi_G([\lambda^{0}_G(f)\star^{1}_{0}j^{1}_{0}(\lambda^{0}_G(f));1^{0}_{1}(\lambda^{0}_G(b))]_{1})
     \\&\qquad\qquad
      \circ^{2}_{1}
     \pi^{0}_G(j^{2}_{1}([\lambda^{0}_G(f)\star^{1}_{0}j^{1}_{0}(\lambda^{0}_G(f));1^{0}_{1}(\lambda^{0}_G(b))]_{1}))
\\
     & =
 \pi^{0}_G([\lambda^{0}_G(f)\star^{1}_{0}j^{1}_{0}(\lambda^{0}_G(f));1^{0}_{1}(\lambda^{0}_G(b))]_{1})
      \\&\qquad\qquad
    \circ^{2}_{1}
     j^{2}_{1}(\pi^{0}_G([\lambda^{0}_G(f)\star^{1}_{0}j^{1}_{0}(\lambda^{0}_G(f));1^{0}_{1}(\lambda^{0}_G(b))]_{1}))
\\
     & =
    1^{1}_{2}(\pi^{0}_G(\lambda^{0}_G (f)\star^{1}_{0}j^{1}_{0}(\lambda^{0}_G (f))))
      \\&\qquad\qquad
     \circ^{2}_{1}
       j^{2}_{1}(1^{1}_{2}(\pi^{0}_G(\lambda^{0}_G (f)\star^{1}_{0}j^{1}_{0}(\lambda^{0}_G (f)))))
 \\
     & =
  1^{1}_{2}(\pi^{0}_G(\lambda^{0}_G(f))\circ^{1}_{0}j^{1}_{0}\pi_G((\lambda^{0}_G (f))))
     \\&\qquad\qquad
     \circ^{2}_{1}
       j^{2}_{1}(1^{1}_{2}(\pi^{0}_G(\lambda^{0}_G (f))\circ^{1}_{0}j^{1}_{0}\pi^{0}_G((\lambda^{0}_G (f)))))
 \\
     & =              
    1^{1}_{2}(1^{0}_{1}(\pi^{0}_G(\lambda^{0}_G (b))))
             \circ^{2}_{1}
         j^{2}_{1}(1^{1}_{2}(1^{0}_{1}(\pi^{0}_G(\lambda^{0}_G (b)))))
 \\
     & =              
   1^{0}_{2}(\pi^{0}_G(\lambda^{0}_G (b)))
             \circ^{2}_{1}
         j^{2}_{1}(1^{0}_{2}(\pi^{0}_G(\lambda^{0}_G (b))))
 \\
     & =              
    1^{0}_{2}(\pi^{0}_G(\lambda^{0}_G (b)))
             \circ^{2}_{1}
         1^{0}_{2}(\pi^{0}_G(\lambda^{0}_G (b)))    
 \\
     & =                       
       1^{0}_{2}(\pi^{0}_G(\lambda^{0}_G (b)))
       =\pi^{0}_G(1^{0}_{2}(\lambda^{0}_G (b)))\,.       
 \end{align*}
The second equality holds because $\pi^{0}_G$ respect the $(\infty,0)$-reversible structures
and the third equality holds because of the definition of an $(\infty,0)$-stretching (see \ref{groupoidal_stretchings}),
the fourth equality is because $\pi^{0}_G$ is a morphism of $\infty$-magmas and $\pi^{0}_G$ preserve the reversible structure,
whereas the fifth equality is because $\pi^{0}_G(\lambda^{0}_G(f))$ and $j^{1}_{0}\pi^{0}_G((\lambda^{0}_G(f))$ 
are $\circ^{1}_{0}$-inverse in the strict $\infty$-groupoid $C^{0}(G)$.
The seventh equality is because $j^{2}_{1}\circ 1^{0}_{2}=1^{0}_{2}$; see \ref{involutive-reflexivity}.
Thus by the contractibility structure in $M^{0}(G)$, we get the following coherence $3$-cell in $M^{0}(G)$:

\[
 \xygraph{
  [] {[\lambda^{0}_G(f)\star^{1}_{0}j^{1}_{0}(\lambda^{0}_G(f));1^{0}_{1}(\lambda^{0}_G(b))]_{1}
 \star^{2}_{1}
 j^{2}_{1}([\lambda^{0}_G(f)\star^{1}_{0}j^{1}_{0}(\lambda^{0}_G(f));1^{0}_{1}(\lambda^{0}_G(b))]_{1})}
 : @3 [d] {1^{0}_{2}(\lambda^{0}_G(b))}  ^{\;\lambda_{f}}
 }
\]
 where $\lambda_{f}$ is the $3$-cell
\begin{align*}
   [[\lambda^{0}_G(f)&\star^{1}_{0}j^{1}_{0}(\lambda^{0}_G(f))\,;1^{0}_{1}(\lambda^{0}_G(b))]_{1}
 \star^{2}_{1}
 j^{2}_{1}([\lambda^{0}_G(f)\star^{1}_{0}j^{1}_{0}(\lambda^{0}_G(f))\,;1^{0}_{1}(\lambda^{0}_G(b))]_{1})
 \\  &  \,;1^{0}_{2}(\lambda^{0}_G(b))]_{2}
 \end{align*}
By applying to it the $\mathbb{T}^{0}$-algebra $(G;v)$ we obtain the following $3$-cell in~$G$
\[
 \xygraph{
  [] {v([\lambda^{0}_G(f)\star^{1}_{0}j^{1}_{0}(\lambda^{0}_G(f));1^{0}_{1}(\lambda^{0}_G(b))]_{1}
 \star^{2}_{1}
 j^{2}_{1}([\lambda^{0}_G(f)\star^{1}_{0}j^{1}_{0}(\lambda^{0}_G(f));1^{0}_{1}(\lambda^{0}_G(b))]_{1}))}
   : @3 [d] {v(1^{0}_{2}(\lambda^{0}_G(b)))}  ^{\;v(\lambda_{f})}
 }
 \]  
with 
 \begin{align*}
 v(&[\lambda^{0}_G(f)\star^{1}_{0}j^{1}_{0}(\lambda^{0}_G(f));1^{0}_{1}(\lambda^{0}_G(b))]_{1}
 \\ &\qquad\qquad\star^{2}_{1}
 j^{2}_{1}([\lambda^{0}_G(f)\star^{1}_{0}j^{1}_{0}(\lambda^{0}_G(f));1^{0}_{1}(\lambda^{0}_G(b))]_{1})
 \\
 &=v([\lambda^{0}_G(f)\star^{1}_{0}j^{1}_{0}(\lambda^{0}_G(f));1^{0}_{1}(\lambda^{0}_G(b))]_{1}) 
    \\ &\qquad\qquad\circ^{2}_{1}  
   v(j^{2}_{1}([\lambda^{0}_G(f)\star^{1}_{0}j^{1}_{0}(\lambda^{0}_G(f));1^{0}_{1}(\lambda^{0}_G(b))]_{1}))
\end{align*}
because $v$ is a morphism of $\infty$-magmas
\begin{align*}
   =v(& [\lambda^{0}_G(f)\star^{1}_{0}j^{1}_{0}(\lambda^{0}_G(f));1^{0}_{1}(\lambda^{0}_G(b))]_{1}) 
    \\ &\qquad\qquad\circ^{2}_{1}  
   i^{2}_{1}(v([\lambda^{0}_G(f)\star^{1}_{0}i^{1}_{0}(\lambda^{0}_G(f));1^{0}_{1}(\lambda^{0}_G(b))]_{1}))
   \\  
  &=v(\beta)
    \circ^{2}_{1}
       i^{2}_{1}(v(\beta))
\end{align*}    
 Thus we obtain the following $3$-cell in $G$ 
      \[\xymatrix{v(\beta)
    \circ^{2}_{1}
       i^{2}_{1}(v(\beta))\ar@3[r]^(.65){v(\lambda_{f})}&\iota^{0}_{2}(b)}\]
But the $\mathbb{T}^{0}$-algebra $(G;v)$ has dimension $2$ thus
 $v(\lambda_{f})$ is an identity, which shows that  
 $v(\beta)\circ^{2}_{1}i^{2}_{1}(v(\beta))=\iota^{0}_{2}(b)$. By the same method 
 we can prove that $i^{2}_{1}(v(\beta))\circ^{2}_{1}v(\beta)=\iota^{1}_{2}(f\circ^{1}_{0}i^{1}_{0}(f))$ which shows that
  $v(\beta)$ is an $\circ^{2}_{1}$-isomorphism. 
  
  Furthermore with the following $2$-cell in $M^{0}(G)$
    \[\alpha=[j^{1}_{0}(\lambda^{0}_G(f))\star^{1}_{0}\lambda^{0}_G(f);1^{0}_{1}(\lambda^{0}_G(a))]_{1}\]  
   we can built a $3$-cell in $M^{0}(G)$ of the type
  
\[
  \xygraph{
   [] {[j^{1}_{0}(\lambda^{0}_G(f))\star^{1}_{0}\lambda^{0}_G(f);1^{0}_{1}(\lambda^{0}_G(a))]_{1}]
         \star^{2}_{1}
 j^{2}_{1}([j^{1}_{0}(\lambda^{0}_G(f))\star^{1}_{0}\lambda^{0}_G(f);1^{0}_{1}(\lambda^{0}_G(a))]_{1}])}
  : @3 [d] {1^{0}_{2}(\lambda^{0}_G(a))} ^{\;\rho_{f}}
  }
\]
and with the same kind of arguments as above we can prove that in $G$ we have the following $2$-cell
  
  \[\xymatrix{i^{1}_{0}(f)\circ^{1}_{0}f\ar@2[r]^(.55){v(\alpha)}&\iota^{1}_{0}(a)}\] 
which is an $\circ^{2}_{1}$-isomorphism, that is we have 
  $v(\alpha)\circ^{2}_{1}i^{2}_{1}(v(\alpha))=\iota^{0}_{2}(a)$ and
      $i^{2}_{1}(v(\alpha))\circ^{2}_{1}v(\alpha)=\iota^{1}_{2}(i^{1}_{0}(f)\circ^{1}_{0}f)$, which finally show that $f$ is an equivalence.  
\end{proof}

\bigbreak{}
  \begin{minipage}{1.0\linewidth}
    Camell \textsc{Kachour}\\
    Department of Mathematics,
    Macquarie University\\
    North Ryde,
    NSW 2109
    Australia.\\
    Phone: 00 612 9850 8942\\
    Email:\href{mailto:camell.kachour@gmail.com}{\url{camell.kachour@gmail.com}}
  \end{minipage}


\begin{thebibliography}{000}

\bibitem{metayerara:browngolasinski} Dimitri Ara and Francois Metayer, \textit{The Brown-Golasinski model structure on strict $\infty$-groupoids revisited}, 
Homology Homotopy Appl (2011), volume 10, pages 121--142.\label{metayerara:browngolasinski}
\bibitem{bat:monglob} Michael Batanin, \textit{Monoidal Globular Categories As a Natural Environment for
		  the Theory of Weak-$n$-Categories}, Advances in Mathematics (1998), volume 136, pages 39--103 .\label{bat:monglob}
\bibitem{sheafi-beke} T. Beke, \textit{Homotopy model categories}, Math.Proc.Camb.Philos.Soc, (2000), volume 129, pages 447--475.\label{sheafi-beke}
\bibitem{boardmanvogt:homotinvarian} J.M. Boardman,  and R.M Vogt, \textit{Homotopy invariant algebraic structures on topological spaces}, Lecture Notes in Mathematics, (1973), volume 347.\label{boardmanvogt:homotinvarian}

\bibitem{francisborceux:handbook2} Francis Borceux, \textit{Handbook of Categorical Algebra}, Cambridge University Press (1994), volume 2.
\label{francisborceux:handbook2}
\bibitem{Dubuc1968} Eduardo Dubuc, \textit{Adjoint triangles}, Lecture Notes in Mathematics (Springer-Verlag 1968), pages 69--91.\label{Dubuc1968}
\bibitem{Cisinsk:Bat} Denis-Charles Cisinski,  \textit{Batanin higher groupoids and homotopy types}, Contemporary Mathematics (2007) 
 pages 171--186.\label{Cisinsk:Bat}
\bibitem{laircoppey:esquisses} Laurent Coppey and Christian Lair, \textit{Le\c{c}ons de th{\'e}orie des esquisses}, Universit{\'e} Paris VII, (1985).
\label{laircoppey:esquisses}

\bibitem{dubuc:kanextensionenriched} Eduardo Dubuc, \textit{Kan extensions in enriched category theory}, SLNM, Springer Verlag (1970), volume 145.
\label{dubuc:kanextensionenriched}


\bibitem{foltz} F. Foltz, \textit{Sur la cat{\'e}gorie des foncteurs domin{\'e}s}, Cahiers de Topologie et de G{\'e}om{\'e}trie
		  Diff{\'e}rentielle Cat{\'e}gorique, XI,2 (1969).\label{foltz}

\bibitem{grothendieck83:_pursuin_stack} Ross Street, \textit{Pursuing Stacks}, Typed manuscript (1983).\label{grothendieck83:_pursuin_stack}

\bibitem{JoyalTierney} Andre Joyal and Myles Tierney, \textit{Quasi-categories vs Segal spaces}, \url{http://arxiv.org/pdf/math/0607820v2} (November 2006).
\label{JoyalTierney}
\bibitem{kamelkachour:defalg} Kamel Kachour, \textit{D{\'e}finition alg{\'e}brique des cellules non-strictes}, Cahiers de Topologie et de G{\'e}om{\'e}trie
		  Diff{\'e}rentielle Cat{\'e}gorique, volume 1 (2008), pages 1--68.\label{kamelkachour:defalg}

\bibitem{metayer:folk} Yves Lafont, Francois Metayer, and Krzysztof Worytkiewicz , \textit{A folk model structure on omega-cat}, Advances in Mathematics (2010), volume
224, pages 1183--1231.\label{metayer:folk}

\bibitem{LurieTopos} Jacob Lurie, \textit{Higher Topos Theory}, \url{http://www.math.harvard.edu/~lurie/papers/highertopoi.pdf} (2011).\label{LurieTopos}

\bibitem{penon1999} Jacques Penon, \textit{Approche polygraphique des $\infty$-cat{\'e}gories non-strictes}, Cahiers de Topologie et de G{\'e}om{\'e}trie
		  Diff{\'e}rentielle Cat{\'e}gorique (1999), pages 31-80.\label{penon1999}

\bibitem{RezkCartesian} Charles Rezk, \textit{A CARTESIAN PRESENTATION OF WEAK \break $n$-CATEGORIES}, 
available online: \url{http://www.math.uiuc.edu/~rezk/cs-objects-revised-arxiv.pdf} (2010).
\label{RezkCartesian}

\bibitem{simpson-conjecture} Carlos Simpson, \textit{Homotopy types of strict $3$-groupoids}, 
\url{http://arxiv.org/pdf/math/9810059v1.pdf} (1998).\label{simpson-conjecture}

\bibitem{simpson-homotopy-higher} Carlos Simpson, \textit{Homotopy theory of higher categories}, 
\url{http://arxiv.org/pdf/1001.4071v1.pdf} (2010).\label{simpson-homotopy-higher}

\bibitem{Smith-Quillen} T Beke, \textit{Sheafifiable homotopy model categories}, Math. Proc. Camb. Philos. Soc (2000), volume 129 (3), pages 447--475.\label{Smith-Quillen}

\end{thebibliography}
\end{document}